\documentclass[final]{amsart}

\usepackage{graphicx}
\usepackage[section]{algorithm}
\usepackage{algorithmic}
\usepackage{enumerate}
\usepackage{mathrsfs,enumerate}
\usepackage{qtree}
\usepackage{multirow,array}
\usepackage{graphicx,subfigure}
\usepackage[dvips]{color}
\usepackage{pstricks,pst-node,pst-tree, pst-plot, pst-eps}

\usepackage[notref,notcite]{showkeys} 

\def\partt#1#2{\frac {\partial #1} {\partial #2}}
\def\cF{{\mathcal F}}
\def\dH{\dot{H}}
\def\cT{{\mathcal T}}
\def\cQ{{\mathcal Q}}
\def\tpsi{{\widetilde \psi}}
\def\Forall{\qquad \hbox{for all }}
\def\hoip#1#2{\langle #1,#2\rangle_1}
\def\hzip#1#2{\langle #1,#2\rangle_0}
\def\hmip#1#2{\langle #1,#2\rangle_{-1}}
\def\dual#1#2{\langle #1,#2\rangle}
\def\beq#1\eeq{\begin{equation} #1 \end{equation}}
\def\bal#1\eal{\begin{aligned} #1 \end{aligned}}

\def\RR{{\mathbb R}}

\chardef\atsign='100

\theoremstyle{plain}
\newtheorem{theorem}{Theorem}[section]
\newtheorem{lemma}[theorem]{Lemma}
\newtheorem{proposition}[theorem]{Proposition}
\newtheorem{corollary}[theorem]{Corollary}

\theoremstyle{theorem}

\newtheorem{remark}{Remark}[section]

\graphicspath{{figures/}}

\begin{document}

\title[Numerical Approximation of Fractional Powers of Elliptic Operators]
{Numerical Approximation of Fractional Powers of Elliptic Operators}

\author{Andrea Bonito}
\address{Department of Mathematics, Texas A\&M University, College Station,
TX~77843-3368.}
\email{bonito\atsign math.tamu.edu}

\author{Joseph E. Pasciak}
\address{Department of Mathematics, Texas A\&M University, College Station,
TX~77843-3368.}
\email{pasciak\atsign math.tamu.edu}

\date{\today}

\begin{abstract} 
We present and study a novel numerical algorithm to approximate the 
action of $T^\beta:=L^{-\beta}$ where $L$ is a symmetric and positive
definite unbounded operator on a Hilbert space $H_0$.
The numerical method is based on a representation formula for
$T^{-\beta}$  
in terms of   Bochner integrals
 involving
$(I+t^2L)^{-1}$ for $t\in(0,\infty)$. 

To develop an approximation to $T^\beta$, we introduce a finite element
approximation $L_h$ to $L$ and base our approximation to $T^\beta$ on
$T_h^\beta:= L_h^{-\beta}$. The direct evaluation of $T_h^{\beta}$
is extremely expensive as it involves expansion in the basis of 
eigenfunctions for $L_h$. The above mentioned representation formula
holds for $T_h^{-\beta}$ and we propose three quadrature  approximations denoted generically by
$Q_h^\beta$.   The two  results of this paper bound the errors
in the $H_0$ inner product of 
$T^\beta-T_h^\beta\pi_h$ and $T_h^\beta-Q_h^\beta$ 
where $\pi_h$ is the $H_0$ orthogonal projection into the finite
element space.
We note that the evaluation of $Q_h^\beta$ involves application of
$(I+(t_i)^2L_h)^{-1}$ with $t_i$ being either a quadrature point or its
inverse.
Efficient solution algorithms for these
problems are available and the problems at different quadrature points
can be straightforwardly solved in parallel.
Numerical experiments illustrating the theoretical estimates 
are provided for both the
quadrature error $T_h^\beta-Q_h^\beta$ and the finite element error
$T^\beta-T_h^\beta\pi_h$. 
\end{abstract} 

\subjclass[2010]{35S15, 65R20, 65N12,	 65N50, 65N30.}

\maketitle

\section{Introduction.}

The mathematical study of  integral or nonlocal operators has received
much attention due to their wide range of applications, see for instance
\cite{ISI:000175019600004,MR2450437,MR0521262,MR2223347,MR2480109,MR1918950,MR660727,MR1727557,MR1709781}.
A prototype for such nonlocal operator results from the fractional
powers of an elliptic operator and is the focus of this paper.    A
canonical example is given by $L^\beta=(I-\Delta)^\beta$ on
$\RR^d$ and is naturally defined via Fourier transform, i.e., 
$$\cF(L^{\beta}f)(\zeta)= \cF((I-\Delta)^\beta f)(\zeta) = (1+|\zeta|^2)^\beta \cF(f)(\zeta).$$
Here $\cF(v)$ denotes the Fourier transform of $v$ and $\beta$ is any
real number.

We shall be concerned with similar problems, but posed on a bounded Lipschitz
domain $\Omega$.      Specifically, given $f$, we seek $u=L^{-\beta} f$
with $\beta\in (0,1)$ and $L$ given by, for example, $L=(I-\Delta)$. 
In this case, we need to augment the problem with
boundary conditions, e.g.,
\beq
\bal \partt{u}{n} &= 0, \qquad \hbox{on }\Gamma_N \subset \partial \Omega,\\
u&=0,\qquad \hbox{on } \Gamma_D:=\partial \Omega\setminus \Gamma_N.\eal
\label{e:bc}
\eeq 
Here $n$ denotes the outward pointing normal on $\Gamma_N$.
When $\Gamma_D$ has positive measure, we may also consider
$L=-\Delta$.
   
We are concerned with the numerical approximation of $u:=
L^{-\beta} f$.  This problem has been studied by numerous authors 
\cite{MR2354493,Gavrilyuk1,Gavrilyuk2,GHK2004,Abner,MR2252038,
MR2300467,MR2800568} (we give more details below). 
We note that
$L^{-\beta}$ is defined in terms of the spectral decomposition of $L$, i.e,  
$u:=L^{-\beta} f $ is given by
\begin{equation}\label{e:intro_spec}
u(x)= \sum_{i=1}^\infty \lambda_i^{-\beta} c_i \psi_i(x).
\end{equation}
Here $c _i$ are the Fourier coefficients of $f$ in the $L_2$-orthonormal
basis $\{ \psi_i(x)\}_{i=1}^\infty $ of eigenfunctions for  $L$ and  $\{\lambda_i\}_{i=1}^\infty \subset \mathbb R_+ \setminus \{ 0 \}$ are the respective eigenvalues.

There are several techniques for approximating 
$L^{-\beta} f$ in the literature.   The most natural approach is to
introduce a (finite dimensional) numerical approximation $L_h$ of $L$ and
obtain an approximation $u_h:=L^{-\beta}_h f$  using an expression
similar to \eqref{e:intro_spec} but  involving  the eigenfunctions and
eigenvalues of $L_h$ \cite{MR2300467,MR2252038,MR2800568}.
The difficulty with this approach is that the direct evaluation of the
resulting approximation is  computationally very expensive.
Indeed, it requires the computation of all eigenvectors and eigenvalues
of a large system of linear equations.  The purpose of this paper is to
provide a more efficient technique for approximating $u_h$  (and
hence $u$) which avoids the above mentioned eigenvalue/eigenvector computations. 

An alternative approach for approximating
$u:=L^{-\beta}f:=(-\Delta)^{-\beta}f$ 
(or a more general  elliptic partial differential equation) 
is based on a representation of $u$ via a ``Neumann to Dirichlet'' map \cite{MR2354493}.
Namely, $u$ is
given, up to a normalization constant, by $u(x)=v(x,0)$ where  
$v: \Omega \times \mathbb R_+ \rightarrow \mathbb R$ solves
$$
-\textrm{div}(y^{1-2\beta} \nabla v(x,y)) = 0, \qquad (x,y) \in \Omega \times \mathbb R_+
$$ 
together with the boundary conditions:
$$ \bal
v(\cdot,y)\hbox{ satisfies }&\hbox{\eqref{e:bc} for each }y\in \mathbb R_+,\\
\lim_{y\rightarrow \infty} v(x,y)&=0,\quad x\in \Omega,\\
\lim_{y\to 0^+} \big(-y^{1-2\beta}  v_y(x&,y)\big) = f(x),\quad
x\in \Omega.
\eal
$$ 
The above equations leads to a well posed variational problem in
an appropriately weighted Sobolev space \cite{Abner} which is amendable
to finite element approximation.
The numerical algorithm proposed and analyzed in \cite{Abner} consists
of a $\mathbb R^{d+1}$ finite element method which takes advantage of
the rapid decay of the solution $v$ in the $y$ direction enabling
truncation to a bounded domain of modest size.
Error estimates are derived in a weighted $H^1(\Omega)$-norm.

Another natural approach for representing $u:=L^{-\beta}f$ is via the
Dunford Taylor integral, i.e., 
\beq
u = \frac 1 {2\pi i} \bigg(\int_\gamma z^{-\beta} R_z(L) \, dz \bigg)f.
\label{e:DunT}
\eeq
Here $\gamma$ is a positively oriented simple closed curve (in the
extended complex plane) enclosing the
spectrum of $L$ and $R_z(L)=(zI-L)^{-1}$ is the resolvent. 
A quadrature approximation of the above integral is proposed in 
\cite{GHK2004}.  In fact, the quadrature proposed in \cite{GHK2004} 
but based on the hyperbolic path given in \cite{Gavrilyuk3} 
leads to an exponentially convergent quadrature approximation.
Our third scheme below is an example of another exponentially convergent
quadrature but, in our case, only involves multiple evaluations
of $(I+t_iL)^{-1}f$ where $t_i\in(0,\infty)$ is related to the quadrature
node.  There are numerous other papers proposing exponentially
convergent quadrature schemes 
for Dunford-Taylor integral representations of $e^{-tL}$, e.g., 
\cite{macleanthomee,hackbush,GHK2004}.

The algorithm developed here is also based on an integral representation
of $u=L^{-\beta}f$:
\beq\bal 
L^{-\beta}  &= \frac{2\sin(\pi \beta)}{\pi}  \bigg[ 
\int_0^\infty t^{2\beta-1} (I+t^2 L)^{-1}   \ dt
\bigg] \\
&=\frac{2\sin(\pi \beta)}{\pi}  \bigg[ 
\int_0^1 t^{2\beta-1} (I+t^2 L)^{-1}   \ dt +\int_0^1 t^{1-2\beta} (t^2I+ L)^{-1}   \ dt
\bigg],
\eal
\label{e:tbeta_intro}
\eeq 
We derive this identity  in Section~\ref{Irep} from  an identity in 
\cite{balakrishnan} 
(see, also 
\cite[Chapter 10.4]{MR1192782} and \cite[Chapter 9.11]{yoshida}).

We develop and analyze two different quadrature schemes for the last two
integrals above and a third for the first expression on the right of
\eqref{e:tbeta_intro} under the transformation $y=\ln(t)$, i.e., 
\beq
L^{-\beta}  = \frac{2\sin(\pi \beta)}{\pi}  \bigg[ 
\int_{-\infty}^\infty e^{2\beta y} (I+e^{2y}  L)^{-1}   \ dy
\bigg].
\label{transformed}
\eeq
The first quadrature is based on a $\beta$-dependent graded
mesh of $N$ subintervals per integral and a one point (rectangle) quadrature on each subinterval.
Theorem \ref{t:fe1} shows that 
there exists a constant $C_1$ only depending on $\beta$ and the smallest eigenvalue of $L$ such that
\begin{equation}\label{e:intro_quad}
\| (L_h^{-\beta} - \cQ_{R,h}^\beta ) f \|_{L^2(\Omega)} \leq C_1 N^{-1} \| f\|_{L^2(\Omega)},
\end{equation}
where $\cQ_{R,h}^\beta$ is the rectangle quadrature approximation to
\eqref{e:tbeta_intro} with $L$ replaced by $L_h$, a finite dimensional approximation of $L$. 
The second quadrature method starts with a dyadic partition of the interval $[0,1]$.
Each subinterval is further decomposed in $N$ intervals of the same length on which a $r$ point weighted Gaussian quadrature is applied.
Overall, this quadrature algorithm uses a number of quadrature points proportional to $r^2 N \ln N$ points. 
Theorem \ref{t:fe1} shows  that there exists a constant $C_2$  only depending on $\beta$ and the smallest eigenvalue of $L$ such that 
\begin{equation}\label{e:intro_quad_gauss}
\| (L_h^{-\beta} - \cQ_{G,h}^\beta ) f \|_{L^2(\Omega)} \leq C_2 N^{-2r} \| f\|_{L^2(\Omega)}.
\end{equation}
Here  $\cQ_{G,h}^\beta$ is the Gaussian quadrature approximation of
\eqref{e:tbeta_intro} with $L$ replaced by $L_h$.
The final quadrature involves $2N+1$ equally spaced points applied to
the integral \eqref{transformed} with a mesh
size $k=1/\sqrt N$ and yields an approximation $\cQ_{E,h}^\beta $ satisfying
\begin{equation}\label{e:intro_quad_expo}
\| (L_h^{-\beta} - \cQ_{E,h}^\beta ) f \|_{L^2(\Omega)} \leq C_3 e^{-c_4/
  k}  \| f\|_{L^2(\Omega)}.
\end{equation}
Here $C_3$ is a constant only depending on $\beta$ and the smallest eigenvalue of $L$ and $c_4$ is an absolute constant.
In what follows we denote $\cQ_{R,h}^\beta$, $\cQ_{G,h}^\beta$ or $\cQ_{E,h}^\beta$ generically by $\cQ_h^\beta$.

The final result  provides estimates for
$\|(L^{-\beta} - L_h^{-\beta} ) f \|_{L^2(\Omega)} $ when $L_h^{-1} f$ is a finite element approximation of $L^{-1} f$.  These estimates
depend on the elliptic regularity index $0<\alpha\le 1$ (see conditions (c)
and (d) of Section~\ref{s:elliptic}), e.g.,  $\alpha =1$ when $\Omega$
is convex, the coefficients of $L$ are smooth 
and $\Gamma_D=\emptyset$ or $\Gamma_N=\emptyset$.
Theorem \ref{t:tbmtbh} shows that under these conditions,
there exists a constant $C_5$ uniform in $h$ and such that
\begin{equation}\label{e:intro_space}\bal
\| (L^{-\beta}-L_h^{-\beta}\pi_h)f \|_{L_2(\Omega)} &\leq C_5
h^{2\alpha}\ln(1/h) \quad  \hbox{or}\\
 \| (L^{-\beta}-L_h^{-\beta}\pi_h)f \|_{L_2(\Omega)} &\leq C_5
h^{2\alpha}\eal
\end{equation}
depending on the smoothness of $f$.  Here $\pi_h$ is now the
$L^2(\Omega)$ projection onto the finite element subspace.
Combining \eqref{e:intro_quad},  \eqref{e:intro_quad_gauss}
or \eqref{e:intro_quad_expo}  with \eqref{e:intro_space} provides an
estimate for the difference between $L^{-\beta}f$ and its numerical
approximation  $Q_h^\beta f$  (see Corollary~\ref{c:final}). 
Although the second result of \eqref{e:intro_space} is well known 
when 
 $\alpha=1$ and $h^{2\alpha}$ is replaced by $h^{2\beta}$ 
(see \cite{FijitaSuzuki} and Remark \ref{FijitaSuzuki}),
we believe that our optimal order results are not available.

The computation of $Q_h^\beta f$ involves multiple
\emph{independent} evaluations of $\big(I+\big(t_i\big)^2
L_h\big)^{-1}f$ with $t_i\in (0,\infty)$
which can be implemented trivially in parallel.
These are just finite element approximations to the problem 
$\big(I+\big(t_i\big)^2 L\big) w=f$ (argumented with boundary
conditions) 
and hence their
implementation is classical.  Moreover, multigrid preconditioners are
available whose convergence rates can be guaranteed to be bounded 
 independently of the
parameters $t_i$ and $h$.  

For most of this work, we provide our analysis in a more general framework,
namely, we assume that we are given two 
Hilbert spaces $H^1 \subset H^0$ with $H^1$ compactly, continuously and densely
embedded in $H^0$.  
Denoting  $H^{-1}$ and $(H^0)'$ to be the respective duals leads to an
operator $T:H^{-1} \rightarrow H^1$ (its inverse is $L$ in the above
discussion).  The usual identification $H^1 \subset H^0 \cong (H^0)' \subset H^{-1}$
enables the definition of $T^\beta$  (see Section \ref{Irep} for the precise setting).
This more general setting contains a large class of operators $T$, for example,
the inverse of the Laplace-Beltrami operator defined over Riemannian surfaces.

The outline of the remainder of the paper is as follows.
In the next section, we motivate the integral identity \eqref{e:tbeta_intro}.
Section \ref{s:num_int} defines our three numerical integration schemes for
the scalar version  of 
 \eqref{e:tbeta_intro} (see \eqref{e:ilam})  and
 provides bounds for the quadrature error.
The main result of this section is Theorem \ref{l:qerror}.  This theorem
is also illustrated by numerical experiments.
The space discretization is developed in Section \ref{s:elliptic} for general elliptic operators. 
Theorem \ref{t:tbmtbh} provides the {\it a-priori} error estimate
\eqref{e:intro_space}.  The theoretical results are again followed by numerical illustration.

\section{Integral Representation of Fractional Powers}
\label{Irep}

We consider two Hilbert spaces $H^1$ and $H^0$  with inner products and
norms, $\langle\cdot,\cdot\rangle_i$ and $\|\cdot\|_i$, $i=0,1$, respectively, and
make the following
two assumptions:
\begin{enumerate} [(a)] 
\item 
$H^1$ is compactly and densely contained in $H^0$. 
\item There is a constant $c_0$ satisfying
\beq
\|u\|_0 \le c_0 \|u\|_1,\Forall u\in H^1.
\label{e:poin}
\eeq
\end{enumerate}

Let $H^{-1}$
denote the dual space of $H^1$ with norm
$$\|F\|_{-1} = \sup_{\phi\in H^1} \frac {\dual F\phi} {\|\phi\|_1},$$
where $\dual\cdot\cdot$ denotes the duality pairing.
We define $T:H^{-1}\rightarrow 
H^1$ by $TF:=u$ where for $F\in H^{-1}$, 
$TF\in H^1$ is the unique solution of 
$$\hoip {TF}\phi  = \dual F\phi ,\Forall \phi\in H^1.$$
It is clear that 
$T$ is a isomorphism of $H^{-1}$ onto $H^{1}$ and we denote $L$ 
to be its inverse.   We define 
$$D(L):= \{ v\in H^0\ : \ Lv\in H^0\}.$$

Our goal is to use \eqref{e:tbeta_intro} to approximate the
fractional powers of $T$, specifically,  $T^\beta f:=L^{-\beta}f$ 
for $0<\beta<1$ and
$f\in H^0$. For convenience we rewrite these integrals as
\beq\bal
T^{\beta}  &= C_\beta^{-1}   
\int_0^\infty t^{2\beta-1}\, T_1(t)  \ dt \\
&= C_\beta^{-1}  \bigg[ 
\int_0^1 t^{2\beta-1}\, T_1(t)  \ dt + \int_0^1 t^{1-2\beta}\, T_2(t) \ dt 
\bigg] 
\label{e:tbeta}
\eal\eeq
with
\begin{equation}\label{e:cbeta}
C_\beta := \int_0^\infty
t^{2\beta-1}(1+t^2)^{-1}dt = \frac{\pi}{2 \sin(\pi \beta)}
\end{equation}
(The last equality can be verified using the Beta function and the Euler reflexion
formula).
Here
 $T_2(t):= t^{-2} T_1(t^{-1})$
and 
$T_1(t):H^{-1}\rightarrow H^1$  is defined by $T_1(t) F := u$ 
with $u\in H^1$ solving
\beq
\hzip u\phi + t^2\hoip u\phi  = \dual F\phi ,\Forall \phi\in H^1.
\label{e:t1}
\eeq
As usual, the  function $f\in H^0$ is identified with the functional $F\in
H^{-1}$ is defined by 
$\dual F\phi = \hzip f\phi$, for all $\phi\in H^1$.
We view \eqref{e:tbeta} as an equality between elements of $B(H^0,H^0)$,
the set of bounded operators on $H^0$. 

We note that 
\beq
\|T_1(t) f\|_0 \le  \|f\|_0 
\label{e:b1}
\eeq
follows immediately from
\eqref{e:t1}.  As $u:=T_2(t)f$ is the solution of 
$$t^2\hzip u\phi + \hoip u\phi  = \hzip f \phi ,\Forall \phi\in H^1,$$
we have from \eqref{e:poin} that
$$\|u\|_0 \le c_0 \|u\|_1 \le c_0 \|f\|_{-1}.$$
It follows, invoking  \eqref{e:poin} again, that for $f\in H^0$, 
\beq
\|f\|_{-1} \le c_0 \|f\|_0
\label{e:hmb}
\eeq
and hence 
\beq
\| T_2(t)f \|_0 = \|u\|_0  \le c_0^2 \|f\|_0.
\label{e:b2}
\eeq

To make sense out of \eqref{e:tbeta}, we derive additional properties of $T$.
The compactness of $H^1$ in $H^0$ and \eqref{e:hmb} immediately imply that
$T:H^{-1} \rightarrow H^{-1}$ is compact.  Moreover, the norm on
$H^{-1}$ comes from the inner product
$$\hmip  FG = \dual F {TG},\Forall F,G\in H^{-1}$$
from which it immediately follows that $T$ is a symmetric and 
positive definite operator on $H^{-1}$.  Fredholm theory implies that
there is an $\hmip \cdot \cdot$ orthonormal basis of eigenfunctions 
$\{\psi_i\}$, $i=1,2,\ldots$ for $H^{-1}$.  The corresponding
(real) eigenvalues $\{\mu_i\}$ are positive and, along with 
their eigenvectors, 
can be reordered to be non-increasing with limit 0. 
Given $F\in H^{-1}$, we 
define 
$$T^\beta F := \sum_{i=1}^\infty \mu_i^\beta \, \hmip F {\psi_i} \, \psi_i.$$

The  functions $\tpsi_i = \mu_i^{1/2} \psi_i$,
$i=1,2,\ldots$, 
provide an orthonormal basis for $H^0$ (see Proposition~\ref{p:scale} below) so that
$$T^\beta f = \sum_{i=1}^\infty \mu^\beta\,  \hzip f {\tpsi_i}\,
\tpsi_i, \Forall f\in H^0.$$
Moreover, 
$$D(L) := \{ u\in H^0\ :\ Lu\in H^0\} = \{ u\in H^0\ : \
\sum_{i=1}^\infty  \mu_i^{-2} \hzip u {\tpsi_i}<\infty\}.$$
This implies that $L$ is a closed operator on $H^0$ with domain 
$D(L)$ and that the range of $L$ on $D(L)$ coincides with $H^0$.
We then have the following theorem.

\begin{theorem} \label{l:1}  Assume that $H^1$ and $H^0 $ are Hilbert 
spaces satisfying (a) and (b).  Then  for $0<\beta<1$,  the integrals
appearing in \eqref{e:tbeta} are Bochner integrable and the equalities
hold in                       $B(H^0,H^0)$.
\end{theorem}

\begin{proof}  We note that if $z=(\lambda + L)^{-1}x$ for  $\lambda>0$
and $x\in H^0$,
$$\|z\|_0 \le \lambda^{-1} \|x\|_0.$$
This implies (cf.  \cite{balakrishnan}, \cite[Chapter 10.4]{MR1192782} and \cite[Chapter 9.11]{yoshida}) that for $\alpha \in
  (0,1)$ and $x\in D(L)$,  the following is a well defined Bochner
  integral and satisfies
$$  \frac {\sin(\alpha \pi) } \pi 
\int_0^\infty \lambda^{\alpha-1} (\lambda I+L)^{-1} Lx \ d\lambda=L^\alpha x.$$
Let $\beta$ be in $(0,1)$.  Applying the above identity with 
$\alpha=1-\beta$ gives
$$ L^{-\beta} y =\frac {\sin(\alpha \pi) } \pi 
\int_0^\infty \lambda^{\alpha-1} (\lambda I+L)^{-1}y \ d\lambda $$
for all $y$ in the range of $L$, i.e., all $y\in H^0$.
Making the change of variable $t^2=\lambda^{-1}$ gives the first
identity in \eqref{e:tbeta}.  Breaking the first integral of \eqref{e:tbeta}
into two integrals on   $(0,1)$ and $(1,\infty)$, respectively, 
and making the change of variables $t\rightarrow t^{-1}$ in the second 
gives the second expression of \eqref{e:tbeta}.
\end{proof}

\begin{remark}  We note that \eqref{e:b2} is equivalent to 
$$\|T_1(t)f\|_0 \le (c_0/t)^2 \|f\|.$$
Using this, \eqref{e:b1} and \eqref{e:b2}, it 
readily follow that the integrals appearing in \eqref{e:tbeta} 
are all well defined Bochner integrals.
\end{remark}

\begin{remark}  It is possible to extend these results to obtain
  negative powers of 
  coercive 
(non-symmetric) bilinear forms but this extension is beyond
the scope of the present paper.
\end{remark}

\section{Numerical Integration.}\label{s:num_int}

The numerical approximation to the integrals appearing in \eqref{e:tbeta}
is based on quadrature approximations to the scalar integrals:
\beq\bal
I(\lambda)&=\int_0^\infty t^{2\beta-1} (1+t^2\lambda)^{-1}dt\\
&= \int_0^1 t^{2\beta-1} (1+t^2\lambda)^{-1}dt + \int_0^1  t^{1-2\beta}
(t^2+\lambda)^{-1}dt=: I_1(\lambda)+I_2(\lambda).
\eal
\label{e:ilam}
\eeq
In the next sections, we develop and analyze three quadratures.

For latter reference, we define 
$$
f_\lambda(t)\equiv 
f^1_\lambda(t):= (1+\lambda t^2)^{-1}, \qquad \text{and} \qquad f^2_\lambda(t):= (t^2+\lambda)^{-1}.
$$

\subsection{Rectangle quadrature rule on graded partitions} \label{ss:rectangle}

We start with $I_1(\lambda)$.
We propose  a graded partition
of $[0,1]$ 
to cope with the singular behavior of $t^{2\beta-1}$ coupled with a one
point quadrature.
For a given positive integer $N$, we set 
\begin{enumerate} [(i)]
\item $t^N_{1,i}:=(i/N)^{\frac{1}{2\beta}}$ if $2\beta-1< 0$.
\item $t^N_{1,i}:=i/N$ if $2\beta-1\ge 0$.
\end{enumerate}
Case (i) above corresponds to refinement at $t=0$ and is necessary to
deal with the singularity of the integrand at zero.  We note that it is
possible to use the node formula of Case (i) even when $2\beta-1>0$.
This results in larger spacing of the nodes near $t=0$
compared to the uniform case and appears to be unnecessary.

The quadrature formula approximating $I_1(\lambda)$ reads
\begin{equation}\label{e:quad}
I_1^N(\lambda):=\sum_{i=0}^{N-1}
f^1_\lambda\big(t_{1,i}^{N,*}\big) 
\int_{t_{1,i}^N}^{t_{1,i+1}^N} t^{2\beta-1} =\sum_{i=0}^{N-1}
\frac{\big (t_{1,i+1}^N\big )^{2\beta} -\big(t_{1,i}^N\big)^{2\beta}
}{2\beta } 
f^1_\lambda\big(t_{1,i}^{N,*}\big).
\end{equation}
Here $t_{1,i}^{N,*}$ is any point in the interval
$(t_{1,i}^N,t_{1,i+1}^N]$.
In Case (i), the coefficient on the rightmost sum of \eqref{e:quad} 
simplifies to $ (2\beta
N)^{-1}$.  

We proceed similarly for $I_2(\lambda)$.  
In this case, the partitioning is given by
\begin{enumerate} 
\item [(iii)] $t^N_{2,i}:=(i/N)^{\frac{1}{2-2\beta}}$ if $1-2\beta< 0$.
\item [(iv)] $t^N_{2,i}:=i/N$ if $1-2\beta\ge 0$.
\end{enumerate}

We define
\beq\bal
I_2^N(\lambda):&=\sum_{i=0}^{N-1}
f^2_\lambda(t_{2,i}^{N,*}) 
\int_{t_{2,i}^N}^{t_{2,i+1}^N} t^{1-2\beta} \\
&=
\sum_{i=0}^{N-1} 
\frac{\big (t_{2,i+1}^N\big )^{2-2\beta} -\big(t_{2,i}^N\big)^{2-2\beta}
}{2-2\beta }
f^2_\lambda(t_{2,i}^{N,*}).\eal
\label{e:quad2}
\eeq
Now, $t_{2,i}^{N,*}$ is any point in the interval $(t_{2,i}^N,t_{2,i+1}^N]$.

Before estimating the quadrature error, we provide a bound on the ratio of two consecutive intervals.
\begin{lemma}[Subdivision Regularity]\label{l:mesh_reg}
Let $N>0$ and $t_i=(i/N)^\theta$ for $\theta > 0$, $i=0,...N$. 
Then, there exists constant $0<\rho(\theta)<\infty$ only depending on $\theta$ such that
\beq
\max_{i=1,..,N-1} \frac{t_{i+1}-t_i}{t_i-t_{i-1}} \leq \rho(\theta).
\label{e:lem31}
\eeq
\end{lemma}
\begin{proof}
For $\theta=1$, the above fraction is identically 1 and hence
we take $\rho(1)=1$.   
For $\theta\neq 1$, we consider the function 
$$
g(t):= \frac{t^{\theta}-(t-1)^{\theta}}{(t-1)^{\theta}-(t-2)^{\theta}}.
$$
Note that $g(i+1)$ coincides with the fraction in \eqref{e:lem31}.   
Moreover, a simple computation shows that 
$g'(t)=0$ if and only if 
$$   (t-1)^{1-\theta} = \frac 1 2  (t^{1-\theta}+(t-2)^{1-\theta}).
$$
The strict concavity of $\theta<1$ and  strict convexity for $\theta
> 1$ of $t^{1-\theta}$ 
imply that $g'(t)$ cannot be zero for any $t$  and so $g(t)$ is monotone.
Hence we can take
\beq
\rho(\theta):= \sup_{t\in [0,\infty]} g(t) = \left\lbrace
\begin{array}{ll}
g(2) = {2^\theta-1} & \qquad \text{when} \qquad \theta >1, \\
\lim_{t\to\infty}g(t) = 1 & \qquad \text{when} \qquad \theta \leq 1, 
\end{array}
\right.
\label{rho}
\eeq
which ends the proof.
\end{proof}

We now provide an estimate on the quadrature error
$$
e(\lambda):=I_1(\lambda)-I_1^N(\lambda)+I_2(\lambda)-I_2^N(\lambda).
$$

\def\tois{t_{1,i}^{N,*}}
\begin{lemma}[Quadrature Error] \label{l:qerror}
For any $0<\beta<1$ and $\lambda>0$, we have
$$
|e(\lambda)| \leq \frac {  2+\rho(A^{-1}) \pi} {N A} +
\frac {  2+\rho(B^{-1})\pi} {\lambda N 
B} 
$$
where $\rho(\cdot)$ is given by \eqref{rho}. Here 
$A=2\beta$ when we use (i) and $A=1$ when $2\beta-1>0$  and we use (ii) .
Similarly, $B=2-2\beta$ when we use (iii) and $B=1$ when $1-2\beta>0$ and we use (iv).
 \end{lemma}
\begin{proof}
Let us start by examining $e^1(\lambda):=I_1(\lambda)-I_1^N(\lambda)$.
We first observe that when we use (i),  the mean value theorem implies
\begin{equation}\label{e:quad_error}\bal
\bigg |\int_{t_{1,i}^N}^{t_{1,i+1}^N} t^{2\beta-1} (f_\lambda(t) &-
f_\lambda\big(\tois\big))\ dt\bigg| \\
&\leq \max_{s,t\in (
  t_{1,i}^N,t_{1,i+1}^N]} |f_\lambda(t)-f_\lambda(s)|   \int_{t_{1,i}^N}^{t_{1,i+1}^N} t^{2\beta-1}\\
&\leq
\frac{1}{ NA} (t_{1,i+1}^N-t_{1,i}^N)\  \| (f_\lambda)'
\|_{L_\infty([t_{1,i}^N,t_{1,i+1}^N])}.
\eal
\end{equation}
In contrast, if $2\beta-1>0$ and we use (ii),  the mean value theorem 
implies that 
\beq
\frac {\big (t_{1,i+1}^N\big )^{2\beta}
  -\big(t_{1,i}^N\big)^{2\beta}}{2\beta}\le \frac {\big (t_{1,i+1}^N\big
  )^{2\beta-1}}N \le \frac1N.
\label{e:mvt}
\eeq
and hence
\beq \label{e:quad_error2}
\bal \bigg |\int_{t_{1,i}^N}^{t_{1,i+1}^N} t^{2\beta-1} (f_\lambda(t)& -
f_\lambda\big(\tois\big))\ dt\bigg|\\
&\leq \frac{\big (t_{1,i+1}^N\big
  )^{2\beta} -\big(t_{1,i}^N\big)^{2\beta} 
}{2\beta } \max_{s,t\in (
  t_{1,i}^N,t_{1,i+1}^N]} |f_\lambda(t)-f_\lambda(s)|  
\\& \le  \frac{1}{NA} 
 (t_{1,i+1}^N-t_{1,i}^N)\  \| (f_\lambda)'
 \|_{L_\infty([t_{1,i}^N,t_{1,i+1}^N])}
\eal
\eeq

We claim that for any $t>0$,
$$
|(f_\lambda)'(t)| = \frac{2\lambda t}{(1+\lambda t^2)^2} \leq \frac{2\lambda^{1/2}}{1+\lambda t^2}.
$$
To see this, we distinguish two cases.
If $\lambda^{1/2}t \leq 1$ then 
$$
|(f_\lambda)'(t)| \leq \frac{2\lambda^{1/2}}{(1+\lambda t^2)^2}  \leq \frac{2\lambda^{1/2}}{1+\lambda t^2}.
$$
Alternatively, if $\lambda^{1/2} t >1$ then
$$
|(f_\lambda)'(t)| \leq 2\lambda^{1/2} \frac{\lambda t^2}{(1+\lambda t^2)^2} = 2\lambda^{1/2} \left( \frac{1}{1+\lambda t^2}-\frac{1}{(1+\lambda t^2)^2} \right)
\leq \frac{2\lambda^{1/2}}{1+\lambda t^2}.
$$
Thus,
\beq
\|( f_\lambda)' \|_{L_\infty([t_{1,i}^N,t_{1,i+1}^N])} 
\leq  \frac{2\lambda^{1/2}}{1+\lambda (t_{1,i}^N)^2}.
\label{e:linf}
\eeq

It follows from  \eqref{e:quad_error}, \eqref{e:quad_error2} and \eqref{e:linf} that 
$$
|e^1(\lambda)|
\le |e^1_0(\lambda)|+
\frac{2}{ NA} \sum_{i=1}^{N-1}  (t_{1,i+1}^N-t_{1,i}^N)
\frac{\lambda^{1/2}}{1+\lambda (t_{1,i}^N)^2}.
$$
where 
$$e^1_0(\lambda):=\int_0^{t_1} t^{2\beta-1}\big(f_\lambda(t) -
f_\lambda\big(t_{1,0}^{N,*}\big)\big)\ dt.$$
Setting  $x_i=\lambda^{1/2} t_{1,i}^N$ gives
$$
|e^1(\lambda)| \leq   
 |e_0^1(\lambda)|+
\frac{2}{NA} \sum_{i=1}^{N-1} 
\frac{x_{i+1}-x_i}{1+x_i^2}.
$$

By Lemma~\ref{l:mesh_reg}, 
$$\max_{i=1,..,N-1} \frac{x_{i+1}-x_i}{x_i-x_{i-1}}=
\max_{i=1,..,N-1} \frac{t_{i+1}-t_i}{t_i-t_{i-1}} \leq \rho(A^{-1}).
$$
Thus,
\begin{align*}
|e^1(\lambda)| 
&\leq  |e^1_0(\lambda)|+\frac{2\rho(A^{-1})}{ NA} \sum_{i=1}^{N-1}
\frac{x_{i}-x_{i-1}}{1+x_i^2}
\\
&\leq
|e^1_0(\lambda)|
+\frac{2\rho(A^{-1})}{NA} \int_0^\infty \frac{1}{1+x^2}dx\\
&=
|e^1_0(\lambda)|
+\frac{2\rho(A^{-1})}{NA} \pi
\end{align*}
where the second inequality followed from the 
 monotone decreasing property of
$1/(1+x^2)$.

The error on the first interval is bounded by
\beq\label{e:fi}
\bal
|e^1_0(\lambda)|&\le  \int_0^{t_{1,1}^N}
t^{2\beta-1}\big(|f_\lambda(t)| +
|f_\lambda\big(t_{1,0}^{N,*}\big)|\big)\, dt\\
&\leq 2 \int_0^{t_{1,1}^N} t^{2\beta-1} \le
 \frac{2}{NA}.
\eal
\eeq
The last inequality followed from 
\eqref{e:mvt} when $2\beta-1>0$ and we use  (ii) and is an obvious equality in Case
(i).
 
Gathering the above estimates, we arrive at
\beq
|e^1(\lambda)| \leq \frac  {2+2\rho(A^{-1}) \pi} {NA}.
\label{e:e1}
\eeq

To bound $e^2(\lambda):=I_2(\lambda)-I_2^N(\lambda)$, we first observe that
$$I_2(\lambda)=\lambda^{-1}\int_0^1  t^{1-2\beta}
f_{\lambda^{-1}}(t)\, dt.
$$
and
$$
I_2^N(\lambda)
=\lambda^{-1}\sum_{i=0}^{N-1}
f_{\lambda^{-1}} \big(t_{2,i}^{N,*}\big) 
\int_{t_{2,i}^N}^{t_{2,i+1}^N} t^{1-2\beta}\, dt.
$$
We see that the quadratures $I_1^N$ and $I_2^N$ are similarly constructed
only differing in that the partitioning is tailored to the power of
$t$ in the approximated integral.   
All of the above arguments apply to $\lambda (I_2-I_2^N)$ and 
lead to 
\beq
\lambda |e^2(\lambda)| \leq \frac {2+2\rho(B^{-1}) \pi}  {NB}.
\label{e:e2}
\eeq
Combining \eqref{e:e1} and \eqref{e:e2} completes the proof of the 
lemma.
\end{proof}

To illustrate the sharpness of the quadrature error estimate
(Lemma~\ref{l:qerror}), we (approximately) compute 
$\|e(\lambda)\|_{L^\infty(10,\infty)}$ by the following algorithm:
\begin{enumerate} [(I)]
\item We examine the value of $|e(\lambda_i)|$ for $\lambda_i=10 \mu^i$ for
$i=0,1,\ldots M$.   Here $\mu >1$ and $M$ is chosen so large that
  $|e(\lambda_i)|$ appears monotonically decreasing for larger $i$.
In any event, $\lambda_M$ is on the order of $10^7$ or larger, a value
which is well beyond our range of applications.
\item  We then refine in a neighborhood of the $\lambda_j$ which obtains the 
maximum of $|e(\lambda_i)|$, $i=0,\ldots,M$ and select the maximizing
$\lambda$ from the refined values.
\end{enumerate}  
Proving the convergence of such a procedure is beyond the scope of this
paper however the plots of $|e(\lambda_i)|$ appear to be sufficiently
well behaved and suggest that the algorithm works.

The values of $\|e(\lambda)\|_{L^\infty(10,\infty)}$  (as computed
above) are reported in Table~\ref{t1} as a function of
$\beta=1/4,1/2,3/4$ and $N$.  For these runs, we chose
$t_{j,k}^{N,*}$  to be the node corresponding to a one point weighted 
Gaussian quadrature.  Specifically, 
$$Q_{1,k}(f)= (t_{1,k+1}^{N}-t_{1,k+1}^{N}) f\big(t_{1,k+1}^{N,*}\big) 
\approx \int^{t_{1,k+1}^{N}}_{t_{1,k+1}^{N}} t^{1-2\beta} f \, dx$$
and
$$Q_{2,k}(f)= (t_{2,k+1}^{N}-t_{2,k+1}^{N}) f\big(t_{2,k+1}^{N,*}\big) 
\approx \int^{t_{2,k+1}^{N}}_{t_{2,k+1}^{N}} t^{2\beta-1} f \, dx$$
with $t_{j,k}^{N,*}$ chosen to make the schemes exact for cubics.

In the case of $\beta=.5$, there is no distinction between the cases
(i)-(ii) and (iii)-(iv) since all result in equally spaced meshes.   
When $\beta=.25$, we used (i) and (iv) for the respective integrals.
We note, however, that switching from (iv) to (iii) in this case made 
negligible  difference in the observed maximum errors. 
For $\beta=.75$, the first column reports (ii)-(iii). 
In this case, switching to (i)-(iii) leads to a slower convergent first order method.


\begin{table}
\begin{tabular}{|c|c|c|c|c|c|} \hline
$N$ & $\beta=.5$&$\beta=.75$
&$\beta=.25$\\ 
\hline
$31      $&$ 2.86\times 10^{-3} $&$	1.60\times 10^{-4}$&$7.77\times 10^{-3}$ \\ \hline	
$63	  $&$ 1.40\times 10^{-3} $&$	5.51\times 10^{-5}$&$	3.82\times 10^{-3}$\\ \hline	
$127 	  $&$ 6.98\times 10^{-4} $&$	1.93\times 10^{-5} $&$1.89\times 10^{-3}$\\ \hline	
$255	  $&$ 3.45\times 10^{-4} $&$	6.75\times 10^{-6}$&$	9.45\times 10^{-4}$\\ \hline	
$511	  $&$ 1.73\times 10^{-4} $&$	2.39\times 10^{-6}$&$	4.71\times 10^{-4}$\\ \hline	
$1023	  $&$ 8.66\times 10^{-5} $&$	8.38\times 10^{-7}$&$	2.35\times 10^{-4}$
\\ \hline
\end{tabular}
\caption{$\|e(\lambda)\|_{L^\infty(10,\infty)}$ for different $\beta$
  (rectangular scheme).
The observed order is $N^{-1}$ ($\beta=0.5$ and $\beta=0.25$) and 
 $N^{-3/2}$ $(\beta=0.75)$.}
\label{t1}
\end{table}

%

\subsection{Gaussian quadrature rule on geometric partitioning}\label{ss:gauss}

We start with the numerical integration of $I_1(\lambda)$.
In contrast with the graded subdivision used in Section \ref{ss:rectangle}, we propose  this time a geometrically refined partition of $[0,1]$ 
to cope with the singular behavior of $t^{2\beta-1}$.

For a given positive integer $M$ yet to be chosen, we consider the dyadic partition $[0,1]= \cup_{i=1}^M I_i \cup  [0,2^{-M}] $, where  
$$
I_i = \lbrack 2^{-i},2^{-i+1} \rbrack , \qquad i=1,...,M.
$$
Furthermore, given an integer $N$, each of these intervals $I_i$ are decomposed onto $N$ uniform subintervals $I_{i,j}$, $j=1,...,N$, of length 
\begin{equation}\label{e:length}
| I_{i,j} | = \frac{1}{2^{i}N}, \qquad i=1,...,M, \quad j=1,...,N.
\end{equation}
Given an integer $r\geq 1$, a $r$ points weighted Gaussian quadrature with weight $w(t)=t^{2\beta-1}$ is applied  on each intervals $I_{i,j}$.
We denote by $t_{i,j}^{1,k}$ and by $A_{i,j}^{1,k}$, $1\leq i \leq M$, $1\leq j \leq N$ and $1\leq k \leq r$  the Gaussian quadrature points and weights respectively.
Hence, the quadrature formula on $I_{i,j}$ reads
\begin{equation*}
I_{i,j}^{N,r}(\lambda):=\sum_{k=1}^{r} f^1_\lambda(t_{i,j}^{1,k}) A_{i,j}^{1,k}.
\end{equation*}
and we propose the following expression for the approximation of $I_1(\lambda)$:
$$
I_1(\lambda) \approx I_1^{N,r}(\lambda):= \sum_{i=1}^M \sum_{j=1}^N I_{i,j}^{N,r}(\lambda).
$$
Notice that we dropped the contribution of the first interval $[0,2^{-M}]$.
The quadrature error $e_1(\lambda):=I_1(\lambda)-I_1^{N,r}(\lambda)$ is given by
\begin{equation}\label{e:quad_error_1}
e_1(\lambda) =  \sum_{i=1}^M \sum_{j=1}^N e_{i,j}^1 + \int_0^{2^{-M}} t^{2\beta-1}f_\lambda^1(t)dt
\end{equation}
where $e_{i,j}^1$ are the  quadrature errors over each subintervals $I_{i,j}$ defined as
\begin{equation}\label{e:error_quad}
e_{i,j}^1(\lambda):= \int_{I_{i,j}} t^{2\beta-1} f_\lambda^1(t)dt - \sum_{k=1}^{r} f^1_\lambda(t_{i,j}^k) A_{i,j}^k.
\end{equation}
At this point, it is worth recalling that for some $\xi_{i,j} \in I_{i,j}$, we have the following  expression for the local quadrature errors
\begin{equation}\label{e:error_gauss}
e_{i,j}^1(\lambda) = \frac{1}{(2r)!} \frac{d^{2r}}{dt^{2r}} f_{\lambda}^1(\xi_{i,j}) \int_{I_{i,j}} t^{2\beta-1} (q^1_{i,j})^2(t) dt, 
\end{equation}
where $q^1_{i,j}(t) = \prod_{k=1}^r(t-t_{i,j}^{1,k})$.

We proceed similarly for $I_2(\lambda)$.  
In this case, we set $w(t):=t^{1-2\beta}$ so that the quadrature error for the second integral $I_2(\lambda)$ reads
\begin{equation}\label{e:quad_error_2}
e^2(\lambda)=\sum_{i=1}^M \sum_{j=1}^N e^2_{i,j}+\int_0^{2^{-M}} t^{1-2\beta} f_\lambda^2(t)dt,
\end{equation}
where 
$$
e_{i,j}^2(\lambda):= \int_{I_{i,j}} t^{1-2\beta} f_\lambda^2(t)dt - \sum_{k=1}^{r} f^2_\lambda(t_{i,j}^{2,k}) A_{i,j}^{2,k}.
$$
Here $t_{i,j}^{2,k}$ and $A_{i,j}^{2,k}$ are the Gaussian points and weights.
Again, the local quadrature error $e_{i,j}^2(\lambda)$ satisfies 
\begin{equation}\label{e:error_gauss2}
e_{i,j}^2(\lambda) = \frac{1}{(2r)!} \frac{d^{2r}}{dt^{2r}} f_{\lambda}^2(\xi_{i,j}) \int_{I_{i,j}} t^{1-2\beta} (q^2_{i,j})^2(t) dt
\end{equation}
for some $\xi_{i,j} \in I_{i,j}$ and where $q_{i,j}^2(t):=\prod_{k=1}^r(t-t_{i,j}^{2,k})$.

Notice that the dependency in the number of geometric intervals $M$ is never specified in the above defined quantities. 
In fact, we are now in position to provide its value
\begin{equation}\label{d:M}
M := \left\lceil \frac{r \log_2(N)}{\min(\beta,1-\beta)} \right\rceil.
\end{equation}
A slightly better algorithm is obtained by using a different number of
geometric intervals for each integral, i.e.,
\begin{equation}\label{d:M1}
M_1 := \left\lceil \frac{r \log_2(N)}{\beta} \right\rceil\quad 
\hbox{ and } \quad M_2 := \left\lceil \frac{r \log_2(N)}{1-\beta} \right\rceil.
\end{equation}
As we shall see in the proof of Lemma \ref{l:quad_gauss}, this choice guarantee that contribution 
$$
\int_0^{2^{-M}} t^{2\beta-1} f^1_\lambda(t) dt  \quad \text{and} \quad \int_0^{2^{-M}} t^{1-2\beta} f^2_\lambda(t) dt
$$ 
present in \eqref{e:quad_error_1} and \eqref{e:quad_error_2} decays optimally. 
More precisely, we show that the quadrature error
$$
e(\lambda) := e_1(\lambda)+e_2(\lambda)
$$
decays like $N^{-2r}$ (independently of $\lambda \geq \lambda_0$) when using $rM$ quadrature points per integral as above.

\begin{lemma}[Quadrature Error]\label{l:quad_gauss}
There exists a constant $C_G$ only depending on $r$ such that the quadrature error $e(\lambda)$ using $ \left\lceil  \frac{r^2 N \log_2 N}{\min(\beta,1-\beta)} \right\rceil$ quadrature points satisfies
$$
e(\lambda) \leq C_G \left(\frac{1}{2\beta}+\frac{1}{2(1-\beta)\lambda} \right) N^{-2r}.
$$
\end{lemma}
\begin{proof}
We start with the quadrature approximation of $I_1(\lambda)$ and focus on the subinterval $I_{i,j}$ for some $1\leq i \leq M$ and $1\leq j \leq N$.
Owing to the error representation formula \eqref{e:error_gauss} and the relation \eqref{e:length} providing the length of $I_{i,j}$ , we directly obtain that
$$
e_{i,j}^1(\lambda) \leq \frac{1}{(2r)!} \left| \frac{d^{2r}}{dt^{2r}} f_{\lambda}^1(\xi_{i,j})\right|  \left( \frac 1 {2^{i}}\right)^{2r} N^{-{2r}} \int_{I_{i,j}} t^{2\beta-1} dt.
$$ 
We compute 
$$
\frac{d^{2r}}{dt^{2r}} f_{\lambda}^1(t) = \sum_{r\leq k \leq 2r }  c_{k} \frac{\lambda^k t^{2k-2r}}{(1+t^2\lambda)^{1+k}},
$$
where $c_{k}$ are absolute constants.
Therefore, $\xi_{i,j} \in I_{i,j}$ implies
$$
\frac{d^{2r}}{dt^{2r}} f_{\lambda}^1(\xi_{i,j})\left( \frac 1 {2^{i}}\right)^{2r}  \leq 2^{2r} \frac{d^{2r}}{dt^{2r}} f_{\lambda}^1(\xi_{i,j}) \xi_{i,j}^{2r}
\leq 2^{2r} \sum_{r\leq k \leq 2r } |c_{k}| \frac{\lambda^k \xi_{i,j}^{2k}}{(1+\xi_{i,j}^2\lambda)^{1+k}} \leq 2^{2r}  \sum_{r\leq k \leq 2r } |c_{k}|. 
$$
Hence, setting $C=\frac{2^{2r}}{(2r)!}\sum_{r\leq k \leq 2r } |c_{k}|$, 
$$
e_{i,j}^1(\lambda) \leq C N^{-2r} \int_{I_{i,j}} t^{2\beta-1}
$$
so that
$$
\sum_{i=1}^M\sum_{j=1}^Ne_{i,j}^1(\lambda) \leq C N^{-2r} \int_{2^{-M}}^1 t^{2\beta-1} \leq \frac{C}{2\beta} N^{-2r}.
$$
In view of \eqref{e:quad_error_1}, it remains to bound $\int_0^{2^{-M}} t^{2\beta-1} f^1_\lambda(t)dt$.
Notice that $f_\lambda^1(t) \leq 1$ yields
$$
\int_0^{2^{-M}} t^{2\beta-1} f^1_\lambda(t)dt \leq \frac{1}{2\beta}2^{-2\beta M} \leq \frac{1}{2\beta} N^{-2r},
$$
where we used the definition \eqref{d:M} of $M$ for the last inequality.
The estimate for $e^1(\lambda)$ follows upon gathering the last two estimates.

We now discuss the approximation of $I_2(\lambda)$.
In this case, notice that there holds
$$
\frac{d^{2r}}{dt^{2r}} f_{\lambda}^2(t) = \sum_{r\leq k \leq 2r }  \tilde c_{k} \frac{t^{2k-2r}}{(t^2+\lambda)^{1+k}},
$$
where $\tilde c_{k}$ are absolute constants.
As a consequence,
$$
\frac{d^{2r}}{dt^{2r}} f_{\lambda}^2(\xi_{i,j}) \xi_{i,j}^{2r} \leq \sum_{r\leq k \leq 2r } \frac{1}{\lambda} |\tilde c_{k}|.
$$
and we directly obtain
$$
\sum_{i=1}^M\sum_{j=1}^Ne_{i,j}^1(\lambda) \leq \frac{\tilde C}{2(1-\beta)\lambda} N^{-2r} 
$$
where 
$\tilde C =\frac{2^{2r}}{(2r)!} \sum_{r\leq k \leq 2r }  |\tilde c_{k}|$.
The contribution from the first interval follows noticing that $f_\lambda^2(t) \leq \frac{1}{\lambda}$ and as in the first case, 
$$
\int_0^{2^{-M}} t^{1-2\beta} f_\lambda^2(t) dt \leq \frac{1}{\lambda}\frac{2^{2(1-\beta)M}}{2(1-\beta)} \leq \frac 1 \lambda \frac{N^{-r}}{2(1-\beta)}.
$$
This ends the estimate for $e^2(\lambda)$ and the proof.
\end{proof}

To illustrate numerically the performance of our quadrature scheme. 
We proceed as in Section \ref{ss:rectangle} and compare 
$$
\|e(\lambda) \|_{L_\infty(10,\infty)}
$$
with 
$$
NSYS:=\hbox{ number of system solves.}
$$
The number of geometric intervals for each of the two integrals was
determined by \eqref{d:M1}.
Table \ref{t:gauss} reports the values of $\|e(\lambda) \|_{L_\infty(10,\infty)}$ when using $2$ Gaussian points (i.e. $r=2$) for different values of $N$.
Except for the case of $N=2$ and $\beta=.75$, the errors for the 
geometric scheme were smaller than those of the rectangle quadrature rule.

\begin{table}[h!]
\begin{tabular}{|c|c|c|c|} \hline
$N$ & $\beta=.5(NSYS)$&$\beta=.75(NSYS)$
&$\beta=.25(NSYS)$\\ 
\hline
$2     $&$ 1.37\times 10^{-3}(72) $&$	8.37\times
10^{-4}(92)$&$2.55\times 10^{-3}(92)$ \\ \hline
$4	  $&$ 8.58\times 10^{-5}(208) $&$	4.16\times
10^{-5}({272})$&
$1.58\times 10^{-4}(272)$\\ \hline	
$8 	  $&$ 5.36\times 10^{-6}(544) $&$	4.22\times
10^{-6}(704) $&
$1.00\times 10^{-5}(704)$\\ \hline	
$16 	  $&$ 3.35\times 10^{-7}(1344) $&$	2.05\times
10^{-7}(1760)$&
$	6.22\times 10^{-7}(1760)$\\ \hline
\end{tabular}
\caption{Two points Gaussian quadrature based on a dyadic partition. The observed order of convergence is $N^{-4}$ as predicted by Lemma \ref{l:quad_gauss}.
} \label{t:gauss}
\end{table}

\subsection{An exponentially convergent quadrature.}
\label{ss:exp}

We next introduce a quadrature scheme which provides exponential
convergence to \eqref{transformed}, specifically,  
\beq
\cQ^\beta = \frac{2k\sin(\pi \beta)}{\pi}      \sum_{\ell=-N}^N e^{2\beta y_\ell} (I+e^{2y_\ell}  L)^{-1} 
\label{inforder}
\eeq
with $y_\ell=\ell k$ and $k=1/\sqrt N $.

Let $\cQ(\lambda)$ for $\lambda\ge \lambda_0>0$ be the expression on the right hand side of \eqref{inforder} with $L$ replaced by $\lambda$. 
We apply the classical analysis for these types of quadrature approximations given in \cite{lundbowers}, specifically,
Theorem 2.20 of \cite{lundbowers}.  Estimates for $\|L^{-\beta}-
\cQ^\beta\|$ will follow from estimates for
$|\lambda^{-\beta}-\cQ^\beta(\lambda)|$
which are  uniform for $\lambda\ge \lambda_0$.

We first note that for $\eta\in \mathbb R$, $|\eta|\le\pi/4$ and $\lambda\ge \lambda_0$,
\beq
|(1+e^{2(y+i\eta)} \lambda)^{-1} |\le \left \{ \bal
  1&:\qquad \hbox{ for } y\le 0,\\
e^{-2y}\lambda_0^{-1}&:\qquad\hbox{ for } y>0.
\eal \right . 
\label{dbinf}
\eeq
It easily follows that $g_\lambda(z):=e^{2\beta z} (1+e^{2z}
\lambda)^{-1}$ is analytic in the strip $D_S=\{ z\in \mathbb C,
|\Im(z)|<\pi/4\}$ for every $\lambda>0$.
In addition, \eqref{dbinf} implies that for $|\eta|\le\pi/4$ and $\lambda\ge \lambda_0$,
 \beq
|g_\lambda(y+i\eta)| 
\le \left \{ \bal
  e^{2\beta y}&:\qquad \hbox{ for } y\le 0,\\
e^{-(2-2\beta)y}\lambda_0^{-1}&:\qquad\hbox{ for } y>0.
\eal \right . 
\label{db1inf}
\eeq
This implies that 
\beq\bal
N(D_S)&:= \max_{\lambda\ge \lambda_0} \bigg\{
\int _{-\infty}^\infty( |g_\lambda(y-i\pi/4) | +  |g_\lambda(y+i\pi/4)
|) \, dy\bigg\}  \\& \le \beta^{-1} + ((1-\beta)\lambda_0)^{-1}.\eal
\label{nbinf}
\eeq
The above considerations readily imply that $g_\lambda$ is in $B(D_S)$ 
for $\lambda\ge \lambda_0$, where $B(D_S)$ is the set of functions analytic in $D_S$ that satisfy
$$
\int_{-\pi/4}^{\pi/4} | g_\lambda(t+iy)|dy = \mathcal O(|t|^\alpha), \quad t\to \pm \infty, \quad 0\leq a <1,
$$ 
(see Definition~2.12 of \cite{lundbowers}).    
We can  apply Theorem 2.20 of \cite{lundbowers}  to conclude  
that for $k>0$,
\beq
  \bigg| 
\int_{-\infty}^\infty g_\lambda(y)   \ dy-
k\sum_{\ell=-\infty}^\infty g_\lambda(\ell k) 
\bigg|
 \le
\frac{N(D_S)}{2\sinh(\pi^2/(4k))} e^{-\pi^2/(4k)}.
\label{220res}
\eeq

The following lemma is an immediate consequence of \eqref{220res}
and the obvious estimates,
\begin{align*}
&k  \sum_{\ell=-N-1}^{-\infty} |g_\lambda(\ell k) | 
\le \frac 1 { 2\beta} e^{-2\beta/k},  \\
&k  \sum_{\ell=N+1}^\infty |g_\lambda(\ell k)|
\le\frac 1{(2-2\beta)\lambda_0}
e^{-(2-2\beta)/k}.
\end{align*}

\begin{lemma}[Quadrature Error]   \label{l:exp}
Let $\cQ^\beta(\lambda)$ be defined by 
\eqref{inforder} with $L$ replaced by $\lambda\ge \lambda_0$,
$N>0$ and $k=1/\sqrt{N}$.
Then
$$\bal
 |\lambda^{-\beta} - \cQ^{\beta} (\lambda)|&\le 
\frac{2\sin(\pi \beta)}{\pi}\bigg[
\frac{ N(D_S)}{2\sinh(\pi^2/(4k))} e^{-\pi^2/(4k)}\\
&\qquad +\frac 1 { 2\beta} e^{-2\beta/k} +\frac 1{(2-2\beta)\lambda_0}
e^{-(2-2\beta)/k}\bigg].\eal
$$
\end{lemma}

\begin{remark}[Exponential Decay] \label{rem:MN} The error from the three exponentials above 
can   essentially be equalized by setting 
$$\cQ^{\beta} (\lambda)= \frac{2k\sin(\pi \beta)}{\pi} 
 \sum_{\ell=-M}^N g_\lambda(\ell k)$$
with
$$\pi^2/(2k)\approx 2\beta k M\approx (2-2\beta) kN.$$
Thus, given $k>0$, we set 
$$ M=\bigg\lceil \frac {\pi^2}{4\beta k^2}\bigg\rceil \quad \hbox{ and }\quad
N=\bigg \lceil \frac {\pi^2}{4(1-\beta) k^2}\bigg \rceil  
$$
and get the estimate
$$\bal |\lambda^{-\beta} - \cQ^{\beta} (\lambda)|&\le 
\frac{2\sin(\pi \beta)}{\pi}\bigg[ 
\frac 1 { 2\beta}  +\frac 1{(2-2\beta)\lambda_0}
\bigg] \bigg[ \frac {e^{-\pi^2/(4k)}}{\sinh(\pi^2/(4k))}+
e^{-\pi^2/(2k)} \bigg].
\eal
$$
We note that the right hand side above asymptotically behaves like 
$$\frac{2\sin(\pi \beta)}{\pi} \bigg( \frac 1 \beta +\frac 1
{(1-\beta)\lambda_0} \bigg) e^{-\pi^2/(2k)}$$
as $k\rightarrow 0$. 
\end{remark}

The convergence results for the exponential scheme are illustrated 
in Table~\ref{t:inforder}.  As usual, we report the total number of 
systems needed to be solved ($NSYS$).   This scheme produces  
errors (using a comparable number of iterations) which are smaller than
those of the other two schemes. We use the equalized 
exponential scheme of the previous remark.  

\begin{table}[h!]
\begin{tabular}{|c|c|c|c|} \hline
$k$ & $\beta=.5(NSYS)$&$\beta=.75(NSYS)$
&$\beta=.25(NSYS)$\\ 
\hline
$1    $&$ 2.71\times 10^{-3}(11) $&$	7.62\times
10^{-4}(15)$&$4.77\times 10^{-3}(15)$ \\ \hline
$1/2	  $&$ 2.45\times 10^{-5}(41) $&$	9.15\times
10^{-6}({55})$&
$3.65\times 10^{-5}(55)$\\ \hline
$1/3	  $&$ 1.80\times 10^{-7}(91) $&$	1.01\times
10^{-7}(120) $&
$3.06\times 10^{-7}(120)$\\ \hline	
$1/4 	  $&$ 1.63\times 10^{-9}(159) $&$	8.01\times
10^{-10}(212) $&
$2.29\times 10^{-9}(212)$\\ \hline	
\end{tabular}
\caption{Error for the exponential scheme with $N$ and $M$ as in 
Remark~\ref{rem:MN}.}
\label{t:inforder}
\end{table}

\subsection{Operator approximation, the finite dimensional case.}
\label{op} 

We use the quadrature approximations of the previous section to develop
approximations to the integrals of Theorem~\ref{l:1} in the finite
dimensional case.  Let $H$ be a finite dimensional space of dimension
$K$ and set $H^1=H^0=H$. We assume
that we have two inner products, $\langle \cdot,\cdot \rangle_i$,
$i=1,2$, defined on $H$ satisfying \eqref{e:poin}.

The quadrature described in Sections \ref{ss:rectangle}, \ref{ss:gauss}
and \ref{ss:exp}  yield three different approximations of the $T^\beta$:
\begin{align}
\cQ_R^\beta &:= \frac 1 {2 N C_\beta} 
\sum_{j=1}^N  [ \beta^{-1} T_1\big( t_{1,j}^{N,*}\big)+ (1-\beta)^{-1} 
T_2\big( t_{2,j}^{N,*}\big) ],
\label{e:finitap}\\
\cQ_G^\beta &:= \frac 1 {C_\beta} 
\sum_{i=1}^M\sum_{j=1}^N \sum_{k=1}^r \left( A_{i,j}^{1,k}  T_1\big( t_{i,j}^{1,k}\big)+  
A_{i,j}^{2,k}T_2\big( t_{i,j}^{2,k}\big)  \right),\quad \hbox{and}
\label{e:finitap_gauss}\\
\cQ_E^\beta &:= \frac k {C_\beta} 
\sum_{i=-N}^N e^{2\beta y_\ell} T_1(e^{2y_\ell}), \qquad y_\ell=\ell k,
\qquad k=1/\sqrt N.
\label{e:finitap_exp}
\end{align}

We use Lemmas~\ref{l:qerror}, \ref{l:quad_gauss} and \ref{l:exp}
to obtain:

\begin{theorem} \label{t:fe1} Let $H^1=H^0=H$ with $H$ finite dimensional 
and  $\cQ_R^\beta$, $\cQ_G^\beta$ and $\cQ_E^\beta$  be defined by
  \eqref{e:finitap}, \eqref{e:finitap_gauss} and \eqref{e:finitap_exp},
 respectively. Then 
$$\bal
\|T^\beta-\cQ_R^\beta\|_0&\le B_R:=
\frac {  2+\rho((2\beta)^{-1}) \pi} {2C_\beta N \beta} +
\frac {  \mu(2+\rho((2-2\beta)^{-1})\pi)} {2C_\beta N 
(1-\beta)}, \\
\|T^\beta-\cQ_G^\beta\|_0&\le B_G:=
\frac{C_G}{C_\beta} \left( \frac{1}{2\beta}+\frac{\mu}{2(1-\beta)}
\right) N^{-2r},
\eal
$$
and
$$\bal
\|T^\beta-\cQ_E^\beta\|_0\le  B_E:&=
\frac 1 {C_\beta}
\bigg[
\frac{ N(D_S)}{2\sinh(\pi^2/(4k))} e^{-\pi^2/(4k)}\\
&\qquad +\frac 1 { 2\beta} e^{-2\beta/k} +\frac \mu{(2-2\beta)}
e^{-(2-2\beta)/k}\bigg].\eal
$$
Here $\mu$ denotes the largest eigenvalue  of $T$, $C_G$ denotes the constant
only depending on $r$ appearing in Lemma \ref{l:quad_gauss} and
$k=1/\sqrt N$. 
\end{theorem}

\begin{proof}  We let $\cQ^\beta$ denote either $\cQ_R^\beta$,
  $\cQ_G^\beta$ or $\cQ_E^\beta$.
As in Section~\ref{Irep}, we use the $\hzip \cdot \cdot$ orthogonal 
basis of eigenfunctions $\{\tpsi_i\}$, $i=1,\ldots,K$ and their
non-increasing eigenvalues $\{\mu_i\}$.  
Expanding $f\in H$ as
$$f=\sum_{i=1}^N \hzip f{\tpsi_i}\ \tpsi_i,$$
we find that 
$$(T^\beta-\cQ^\beta)f= C_\beta^{-1} \sum_{i=1}^N 
e(\mu_i^{-1}) 
\hzip f{\tpsi_i}\ \tpsi_i$$
where
$$e(\lambda):=\lambda^{-\beta}-\cQ^{\beta}(\lambda).$$
The theorem follows immediately from Lemmas \ref{l:qerror},
\ref{l:quad_gauss} and \ref{l:exp}.
\end{proof}

\section{Application to an elliptic boundary value problem}
\label{s:elliptic}

We restrict our considerations to $H^0=L_2(\Omega)$ and
$H^1=H^1_0(\Omega)$ with $\Omega$ a bounded polygonal domain in $\RR^d$,
$d=1,2,3$ with a
Lipschitz continuous boundary.  
We denote by $\|.\|_{L_2}$ and $\|.\|_{H^1}$ the standard $L_2$ and $H^1$ norms respectively.
In order to simplify the notations, we write $g_1(\cdot)\preceq g_2(\cdot)$
when 
$g_1(\cdot)\leq C g_2(\cdot)$ for a constant
$C$ independent of the argument in $g_1$ and $g_2$.

Let $A(\cdot,\cdot)$ denote an $H^1$ 
coercive, symmetric bilinear form on $H^1 \times H^1$, e.g.,
$$A(v,w) = \int_\Omega a(x) \nabla v\cdot \nabla w\  dx$$
with $a(x)$ satisfying $0<a_0\le a(x) \le a_1$ for positive numbers 
$a_0,a_1$.  We take 
$$<v,w>_1=A(v,w) \qquad \hbox{and} \qquad <v,w>_0 = \int_\Omega vw\,
dx.$$

Clearly, Condition (a) of Section~\ref{Irep}
holds for this pair of spaces.   
Moreover, Condition (b) is just the Poincar\'e
inequality which holds as well.

Let $H_h\subset H^1$ be the space of continuous piecewise linear finite element functions defined with respect
to a quasi uniform triangulation $\cT_h$ of $\Omega$.  By this we mean
that there exists $\rho>0$ independent of $h$  such that for any $h>0$
\beq
 \max_{T\in \mathcal T_h} \text{diam}(T) \leq 
 \rho  \min_{T\in \mathcal T_h} \text{diam}(T).
\label{e:quasi}
\eeq

For $F\in H^{-1}(\Omega)$, the dual of $H^1=H^1_0(\Omega)$, the finite
element approximation $T_h(F)$ of $T(F)$ is defined by $T_h(F)=u_h\in
H_h$ satisfying 
$$ <u_h,\phi_h>_1 = <F,\phi_h> ,\Forall \phi_h\in H_h.$$

We shall approximate first $T^\beta$ by $T_h^\beta \pi_h$ 
and subsequently apply the results of Section~\ref{op} to develop 
 an approximation to $T_h^\beta\pi_h$.  Let
$T_{1,h}$ be defined by \eqref{e:t1} with $H^1$ replaced by 
$H_h$ and $T_{2,h}(t)= t^{-2} T_{1,h}(t^{-1})$.  
Note that the definitions of $T_{i,h}(t)F $,  $i=1,2$  make 
perfect sense for $F\in H^{-1}$.  
In addition, for  $f\in H^0$, 
\beq
\bal
T_{i,h}(t) \pi_h f= T_{i,h}(t)  f,\qquad \hbox{for }i=1,2.
\label{e:pigo}
\eal
\eeq
In addition, the assumption $H_h \subset H^1$ implies that \eqref{e:b1}
and \eqref{e:b2} also hold for $T_{1,h}(t)$ and $T_{2,h}(t)$.
We define $\cQ_{R,h}^\beta$, $\cQ_{G,h}^\beta$ and $\cQ_{E,h}^\beta$
by replacing $T_1$ and  $T_2$ by $T_{1,h}$ and $T_{2,h}$ 
in \eqref{e:finitap}, \eqref{e:finitap_gauss} and \eqref{e:finitap_exp}.

The goal of this section is to estimate $\|(T^\beta-\cQ_h^\beta\pi_h)f\|_{L_2}$
for appropriate $f$.   Theorem~\ref{t:fe1} immediately leads to a bound
for $\|(T^\beta_h - \cQ_h^\beta)\pi_h\|_{L_2}$ and hence we need only 
estimate  $\|(T^\beta-T_h^\beta\pi_h)f\|_{L_2}$.  This, in turn, will require
more precise regularity estimates for the operator $T$.  To this end, 
we define intermediate spaces
for any $-1 \leq s \leq 2$,
$$
\widetilde{H}^s(\Omega) := \left\lbrace\begin{array}{ll}
H^1_0(\Omega) \cap H^s(\Omega), & \qquad 1\leq s \leq 2, \\
\lbrack L_2(\Omega),H^1_0(\Omega)\rbrack_s, & \qquad 0 \leq s \leq 1, \\
\lbrack H^{-1}(\Omega),L_2\rbrack_{1+s}, & \qquad -1 \leq s \leq 0, \\
\end{array}\right.  
$$
where $\lbrack \cdot ,\cdot \rbrack_s$ denotes the real interpolation method and
$H^{-1}(\Omega)=H^{-1}$ denotes the dual of $H^1_0(\Omega)$.

As already observed in Section~\ref{Irep}, $T:H^1\rightarrow H^{-1}$ 
is an isomorphism, i.e., $T:\widetilde{H}^{1} (\Omega)\rightarrow 
\widetilde{H}^{-1} (\Omega)$ is an isomorphism.  For $s\ge -1$, let
$$\dH^s = \{F\in H^{-1}: \sum_{i=1}^\infty \mu_i^{-s}\dual F{\tpsi_i}^2
<\infty\}
$$
with the natural norm
\begin{equation}
\label{e:norm_dot}
\|F\|_{\dH^s} = \bigg( \sum_{i=1}^\infty \mu_i^{-s}\dual F{\tpsi_i}^2
\bigg)^{1/2}.
\end{equation}
Note that the
definition of these spaces immediately implies that
$T:\dH^s\rightarrow \dH^{s+2} $ is an isometry for any real $s$.
Proposition~\ref{p:scale} below 
shows that $\dH^s$ and $\widetilde{H}^s$ coincide 
for $s\in [-1,1]$ and their norms are equivalent.  

To get the equivalence to extend for $s>1$, we require two additional conditions:
\begin{enumerate}
\item [(c)] There is an $0<\alpha\leq 1$ such that 
$T$ is a bounded map of $\widetilde{H}^{-1+\alpha} (\Omega)$
into $\widetilde{H}^{1+\alpha} (\Omega)$.    
\item [(d)] 
Let $L:H^1(\Omega)\rightarrow H^{-1}(\Omega)$ be defined by
$$\dual {Lu} \phi = A(u,\phi),\Forall \phi\in H^1_0(\Omega).$$
We assume that the restriction of $L$  is  a bounded operator from
$H^{1+\alpha}(\Omega) $ to $\widetilde H^{-1+\alpha}(\Omega)$.
\end{enumerate}
The first is an elliptic regularity
result which can be found in, e.g.,
\cite{kellogg,bramblebacuta,dauge,others}.
The second is somewhat simpler related result.   Both depend on the
smoothness of the coefficients defining $\hoip \cdot\cdot$.

We illustrate the proof of (d) in  the case of the Dirichlet form ($a(x)=1$).
For $u\in H^2(\Omega)$ and $\phi\in H^1_0(\Omega)$,
$$ |\hzip {Lu} \phi| = | A(u,\phi)| =|(-\Delta u,\phi)| \le \|u\|_{H^2(\Omega)} \|\phi\|_{L^2(\Omega)}.$$
This implies that $L$ is a bounded operator from $H^2(\Omega)$ into $L^2(\Omega)$.
  Clearly, $L$ is 
  a bounded operator from $H^1(\Omega)$ into $H^{-1}(\Omega)$.
Interpolating  these results gives 
$$\|Lu\|_{\widetilde H^{-1+s}(\Omega)} \le C\|u\|_{H^{1+s}(\Omega)},\Forall u\in H^{1+s}(\Omega),\
1\le s\le 2.$$
This estimate is a bit more tricky for more general $A(\cdot,\cdot)$, for example,
when $A(\cdot,\cdot) $ involves jumping coefficients.

The above two assumptions imply the following proposition.  We note that the case of $\alpha=1$ 
was given in \cite{thomee}.

\begin{proposition} \label{p:scale}
Assume that (a)--(b) hold.  Then the spaces $\widetilde H^s(\Omega)$ and 
  $\dH^s$ coincide for $s\in [-1,1]$ and their norms are equivalent.   If, in addition,
(c) and (d) hold, then the above equivalence extends to $s\in [1,1+\alpha]$.
\end{proposition}

\begin{proof} As discussed earlier, $\{\psi_i\}$, $i=1,2,\ldots $ is an
  orthonormal basis for $H^{-1}$ with associated eigenvalues $\{\mu_i\}$ and we recall the definition $\tpsi_i:= \mu_i^{1/2} \psi_i$.  
This implies that $F\in H^{-1}$ if and only if the 
series
$$\sum_{i=1}^\infty \hmip F{\psi_i}  \psi_i$$
converges (to $F$) in $H^{-1}$, but
$$\sum_{i=1}^\infty \hmip F{\psi_i} ^2=\sum_{i=1}^\infty \mu_i^2 \dual F{\psi_i} ^2
=\sum_{i=1}^\infty \mu_i \dual F{\tpsi_i} ^2.$$
This shows that $H^{-1}$ and $\dH^{-1}$ coincide and have identical norms.

It is clear that $\{\tpsi_i\}$ is an orthonormal set in $H^0$.  Note that if $w\in H^0$
satisfies $\hzip w{\tpsi_i}=\mu_i^{1/2} \dual w {\psi_i}=0 $ for all
$i$,  then the identified functional 
is zero 
in $H^{-1}$ and hence
$$\hzip w{\phi}=0,\Forall \phi\in H^{1}.$$
The density of $H^{1}$ in $H^0$ implies that $w=0$, i.e., $\{\tpsi_i\}$ 
is an orthonormal basis for $H^0$.  As above, it follows that $H^0$ and $\dH^0$ coincide and have 
identical norms.

The argument for  $H^1$ is essentially the same as for $H^0$ except the
$H^1$-orthonormal  basis is now 
of the form $\{\mu_i  \psi_i\}$, $i=1,2,\ldots$.  This implies that $H^1$ and $\dH^1$ coincide and have 
identical norms.   The proposition then follows for $s\in [-1,1]$ by interpolation.

We finally extend the result to $[1,1+\alpha]$ when  (c) and (d) hold.
Suppose that $u$ is in $\widetilde H^{1+\alpha}(\Omega)$.  Then
$$u=\sum_{i=1}^\infty \hmip  u{\psi_i}  \psi_i$$
where the series converges, at least, in $H^{-1}$.
Now, $TLu=u$ for $u\in \widetilde H^{1+\alpha}(\Omega) \subset H^1_0(\Omega)$ so 
$$\mu_i \dual {Lu} {\tpsi_i} =  \mu_i \hoip  {u} {\tpsi_i} =  \hoip  {u}
{T\tpsi_i} =  \dual u {\tpsi_i} .$$
According to (d) and the first part of this proof,  $Lu\in \widetilde H^{-1+\alpha}(\Omega)=\dH^{-1+\alpha}$ hence
$$\sum_{i=1}^\infty \mu_i^{1-\alpha} \dual {Lu}{\tpsi_i}^2 =\sum _{i=1}^\infty\mu_i^{-1-\alpha} <u,\tpsi_i>^2<\infty,$$
i.e., $u\in \dH^{1+\alpha}$.    Moreover, 
$$\|u\|_{\dH^{1+\alpha}} =\|Lu\|_{\dH^{-1+\alpha}} = \|Lu\|_{\widetilde H^{-1+\alpha}(\Omega)} \le C 
\|u\|_{\widetilde H^{1+\alpha}(\Omega)}.$$
Alternatively, if $u$ is in $\dH^{1+\alpha}$, then $Lu$ is in $\dH^{-1+\alpha}=\widetilde H^{-1+\alpha}(\Omega)$.  Now $u$ solves
$$A(u,\phi)= <Lu,\phi>,\Forall \phi\in H^1$$
and  hence (c) implies that $u$ is in $\widetilde H^{1+\alpha}(\Omega)$ and satisfies
$$\|u\|_{\widetilde H^{1+\alpha}(\Omega)} \le C\|Lu\|_{\widetilde
  H^{-1+\alpha}(\Omega)}=  C \|Lu\|_{\dH^{-1+\alpha}} =
C \|u\|_{\dH^{1+\alpha}}.$$
This shows that the spaces are identical at $s=1+\alpha$.   

We note that the spaces $\widetilde 
H^{1+s}(\Omega)$ for $s\in (0,\alpha)$ are the intermediate spaces on the interpolation scale between 
$\widetilde 
H^{1}(\Omega)$ and $\widetilde 
H^{1+\alpha}(\Omega)$.    This result is a consequence of Lemma A1 of
\cite{guermondima2009}  (see also Lemma A2 there).  The projector needed for
the application of Lemma A1\cite{guermondima2009}  is  $P=TL$ in our context and its 
stability on $H^{1+\alpha}(\Omega)$ is a consequence of (c) and (d).
Thus, the result for $s\in (1,1+\alpha)$ follows by interpolation and
completes the proof of the proposition.
\end{proof} 

The previous result coupled with assumption (c) and  a duality argument entails the following approximation estimate for $T-T_h$.

\begin{corollary}\label{c:apprx}
Assume (a), (b) and (c) hold.
Then there exists a constant $C$ independent of $h$ such that for all $f \in \dH^{\alpha-1}$ there holds
$$
\| (T-T_h) f \|_{\dH^{1-\alpha}} \leq C h^{2\alpha} \| f \|_{\dH^{\alpha-1}}.
$$
\end{corollary}
\begin{proof}
Owing to (c) we deduce that
$$
\| (T-T_h)f\|_{\dH^1} \preceq h^{\alpha} \| T f \|_{\dH^{1+\alpha}} \preceq h^{\alpha} \|f\|_{\dH^{\alpha-1}},
$$ 
where Proposition \ref{p:scale} is used to obtain the last estimate.
The desired estimate follows from a standard duality argument (Nitsche trick).
\end{proof}

The following equivalence relation will be instrumental  in the proof of Theorem \ref{t:tbmtbh}.
For $0\leq s \leq 1$, there exists a constant $c$ independent of $h$ such that for all  $v_h \in H_h$
\begin{equation}\label{p:discretenorm}
\frac 1 c \|v_h\|_{\dH^s_h}   \leq \| v_h \|_{\dH^s} \leq c  \|v_h\|_{\dH^s_h},
\end{equation}
where 
$$
\|v_h\|_{\dH^s_h} :=\bigg( \sum_{i=1}^M \mu_{i,h}^{-s}\dual {v_h} {\tpsi_{i,h}}^2\bigg)^{1/2},
$$
and $\mu_{i,h}$, $i=1,...,M$ are the eigenvalues of $T_h$ with
corresponding $L^2$-orthonormal eigenfunction $\tpsi_{i,h}$ (compare
\eqref{e:norm_dot} with the norm on $\dH^s$).
The above equivalence is a well-known result in the literature on
multigrid methods (see, e.g., \cite[Appendix A.2]{00BZ}).

We are now in position to derive the  following result on the space discretization error.

\begin{theorem}[Convergence of the space-discretization]\label{t:tbmtbh}
Assume that (a)--(d)  and \eqref{e:quasi} hold.  Set 
$\gamma=\alpha-\beta $ when $\alpha\ge \beta $ and $\gamma=0$ when 
$\alpha<\beta$.   For $\delta\ge \gamma$,
There exists a constant $C$ uniform in $h$ and $\delta$ such that 
$$
\| (T^{\beta  } - T_h^{\beta}\pi_h) f\|_{L_2} 
\leq C_{\delta,h}  h^{2\alpha}\|f\|_{\dH^{2\delta}}, \Forall f\in \dH^{2\delta}.
$$
Here
\beq
C_{\delta,h}=\left \{ \bal C\ln(1/h)&:\qquad \hbox{ when } \delta
=\gamma \hbox{ and } \alpha\ge
\beta,\\
C\kern 2em&:\qquad \hbox{ when } \delta>\gamma
\hbox{ and } \alpha\ge
\beta,\\
C\kern 2em&:\qquad \hbox{ when } 
\beta>\alpha.
\eal
\right.
\label{cdelta}
\eeq
\end{theorem}

\begin{remark}[Less regularity: $\delta < \alpha-\beta$] \label{dlessamb}
If $f\in
\dH^{2\delta}$ with $\delta < \alpha-\beta$, $T^\beta f$ is only in
$\dH^{2\beta+2\delta}$ and hence the best possible rate is
$O(h^{2\beta+2\delta})$.  This can essentially be recovered from the
  theorem.  Indeed, if (c) and (d) hold for $\alpha$ then, by interpolation,
they hold for any $\alpha_0$ in $(0,\alpha)$. Applying the theorem with 
$\alpha=\delta+\beta$ gives 
$$\| (T^{\beta  } - T_h^{\beta}\pi_h) f\|_{L_2} 
\leq C \ln(h^{-1})  h^{2\beta+2\delta}\|f\|_{\dH^{2\delta}}, 
\Forall f\in \dH^{2\delta}.
$$
\end{remark}

The identity $\|TF\|_1=\|F\|_{-1}$ combined with \eqref{e:poin} and
\eqref{e:hmb} implies that the largest eigenvalue of $T_h$ is bounded by
$c_0^2$. 
Combining this with the above theorem and Theorem~\ref{t:fe1} gives the following
corollary.

\begin{corollary} \label{c:final}
Assume that (a)--(d)  and \eqref{e:quasi} hold.
Then for $j=R,G,E$,
there exists a constant $C$ uniform in $h$ and $N$ such that for all $
f\in \dH^{2\delta}$ 
$$
\| (T^{\beta  } - \cQ_{j,h}^{\beta}\pi_h)f \|_{L_2} 
\leq C_{\delta,h} 
h^{2\alpha}  \|f\|_{\dH^{2\delta}} +B_j  \|f\|_{L_2}.
$$
Here $\delta$ and $C_{\delta,h} $ are given in the above theorem
and $B_R,B_G,B_E$ are given in Theorem~\ref{t:fe1}.
\end{corollary}

We introduce the following lemma which is needed for the proof of the
above theorem.  Its proof is given after the proof of the theorem.

\begin{lemma} \label{techlem} 
For any $\eta\in [0,(1+\alpha)/2)$ set  $\theta=(1+\alpha)/2-\eta$.
Then
$$\|t^{2\theta} (t^2I+T)^{-1} f \|_{\dH^{\alpha-1}} \le \|f\|_{\dH^{2\eta}},\Forall
f\in \dH^{2\eta}.
$$
In addition, 
$$ \|t^{1+\alpha} (t^2I+T_h)^{-1} f \|_{L^2} \le 
  \|f\|_{\dH_h^{1-\alpha}},\Forall f\in H_h.$$
The above results hold for all $t\ge 0$.
\end{lemma}

\begin{proof}   [Proof of Theorem~\ref{t:tbmtbh}.]
We first note that 
$$(I-\pi_h) (T^{\beta} - T_h^{\beta} \pi_h) f = 
(I-\pi_h) T^{\beta}  f 
$$ 
so that 
$$\|(I-\pi_h) (T^{\beta} - T_h^{\beta} \pi_h) f\|_{L^2} \le 
Ch^{2\alpha} \|T^{\beta} f\|_{H^{2\alpha}(\Omega)}\le Ch^{2\alpha}
  \|T^\beta f\|_{\dH^{2\alpha}}\le  Ch^{2\alpha} \|f\|_{\dH^{2\delta}}.$$
Thus, we need only bound 
$$\|\pi_h (T^{\beta} - T_h^{\beta} \pi_h) f\|_{L^2}.
$$
Applying Theorem~\ref{l:1} and \eqref{e:pigo} 
implies
\begin{equation}\label{e:split}
\begin{split}
\pi_h(T^{\beta} - T_h^{\beta} \pi_h) f=&C_\beta^{-1} \pi_h\int_0^1
t^{2\beta-1} (T_1(t)-T_{1,h}(t))   f \ dt
\\
&+ C_\beta^{-1} \pi_h\int_0^1 t^{1-2\beta}(T_2(t)-T_{2,h}(t))  f \ dt\\
&=: I_1+I_2.
\end{split}
\end{equation}

We first consider $I_1$.  Setting $t_0=h^{\alpha/\beta}$, 
applying \eqref{e:b1}, \eqref{e:b2} and the triangle inequality, we obtain
\begin{equation}
\begin{split}
     \bigg \| \pi_h \int_0^{t_0}
t^{2\beta-1} (T_1(t)-T_{1,h}(t))   f \ dt\bigg \|_{L^2}
&\leq 
2 \bigg(\int_0^{t_0} t^{2\beta-1} \ dt\bigg)\ \| f\|_{L_2}\\&
=
\frac{ h^{2\alpha}}{\beta}   \| f \|_{L_2}.
\label{e:smallt}
\end{split}
\end{equation}
Hence, it remains to derive an upper bound for 
\beq
\bigg \| \pi_h \int_{t_0}^1
t^{2\beta-1} (T_1(t)-T_{1,h}(t))   f \ dt\bigg \|_{L_2}.
\label{it01}
\eeq
It is easy to check that  $T_1(t)=T (t^2I+T)^{-1}$.
Similarly, defining $(t^2I+T_h)^{-1}$  to be the inverse of $t^2I+T_h$ on
$L_2(\Omega)$, we find that
$T_{1,h}(t)=T_h (t^2I+T_h)^{-1}= (t^2I+T_h)^{-1} T_h$ on $L_2(\Omega)$.
Thus,
\beq\bal
\pi_h(T_1(t)-T_{1,h}(t))&=\pi_h(T(t^2I+T)^{-1} -(t^2I+T_h)^{-1}T_h)\\
&=t^{2}(t^2I+T_h)^{-1} \pi_h (T-T_h) (t^2I+T)^{-1}.
\eal
\label{e:tdif}
\eeq

We first consider the case of $\alpha\ge \beta$. 
Note that
$\gamma<(1+\alpha)/2$ and the theorem in this case will 
follow for all $\delta$ if we can prove it for $\delta\in
[\gamma,(1+\alpha)/2)$
which we shall henceforth assume.  

Applying  \eqref{e:tdif} shows that 
\eqref{it01} can be bounded by 
\begin{equation}\label{e:I1}\bal
\int_{t_0}^1 t^{-1+2\delta+2\beta-2\alpha}&\big[
 \|t^{1+\alpha} (t^2 I+
T_h)^{-1}\pi_h\|_{\dH^{1-\alpha}\to L_2} \|(T-T_h)\|_{\dH^{\alpha-1}\to
  \dH^{1-\alpha}}\\
&  \|t^{1+\alpha-2\delta}(t^2I+T)^{-1}\|_{\dH^{2\delta} \to
    \dH^{\alpha-1}}\big] \,dt \  \|f\|_{\dH^{2\delta}}.\eal
\end{equation}
Corollary \ref{c:apprx} shows that 
\beq
\|(T-T_h)\|_{\dH^{\alpha-1}\to
  \dH^{1-\alpha}}\preceq h^{2\alpha} 
\label{tmth}
\eeq
while the first part of Lemma~\ref{techlem} (with $\eta=\delta\in
[\gamma,(1+\alpha)/2)\subseteq [0,(1+\alpha)/2)$) 
implies that the third norm of \eqref{e:I1} is uniformly bounded.  
Appying the second part of Lemma~\ref{techlem} and  the equivalence of
norms \eqref{p:discretenorm} gives
\beq\bal
\|t^{1+\alpha} (t^2 I+
T_h)^{-1}\pi_h f\|_{L^2} \le \|\pi_h f\|_{\dH_h^{1-\alpha}}
\le C \|\pi_h f\|_{\dH^{1-\alpha}}.\eal 
\label{almost}
\eeq
Since $\pi_h$ is bounded as an operator from $L_2(\Omega)$ to
$L_2(\Omega)$ and from $H^1_0(\Omega)$ to $H^1_0(\Omega)$ due to the
quasi-uniformity \eqref{e:quasi} of the triangulations, it follows 
by interpolation that 
$\pi_h$ is also bounded from $\dH^{1-\alpha}$ to $\dH^{1-\alpha}$.
Combining this with \eqref{almost} shows that the first norm of
\eqref{e:I1} is also uniformly bounded.  
It follows that when $\alpha\ge
\beta$, 
$$\bal
\|I_1\|_{L_2} &\le 
\frac{ h^{2\alpha}}{\beta}   \| f \|_{L_2}+C h^{2\alpha} \int_{t_0}^1 t^{-1+2\delta-2\gamma}\, dt\, \|f\|_{\dH^{2\delta}}\\& \le
\left \{ \begin{aligned} C h^{2\alpha} \ln(h^{-1})\|f\|_{\dH^{2\delta}}&:\qquad \hbox{ when }  
\delta=\gamma\\
C h^{2\alpha}\|f\|_{\dH^{2\delta}}\kern 2.0em &:\qquad\hbox{ when
}\delta>\gamma. 
\end{aligned}
\right.\eal$$

We next bound $I_1$ when $\alpha<\beta$. In this case, we bound \eqref{it01} by 
$$\bal 
\int_{t_0}^1 t^{-1+2\beta-2\alpha}&\bigg[
 \|t^{1+\alpha} (t^2 I+
T_h)^{-1}\pi_h\|_{\dH^{1-\alpha}\to L_2} \|(T-T_h)\|_{\dH^{\alpha-1}\to
  \dH^{1-\alpha}}\\
&  \|t^{1+\alpha}(t^2I+T)^{-1}\|_{L_2 \to
    \dH^{\alpha-1}}\bigg] \,dt \  \|f\|_{L_2}.\eal
$$
We  use the second part of Lemma ~\ref{techlem} for the first term,
\eqref{tmth} for the middle term and  the first part of 
Lemma~\ref{techlem} with $\eta=0$ for the for the last. 
The bound
$$\|I_1\|_{L^2} \preceq h^{2\alpha} \|f\|_{L^2} 
$$ 
follows immediately.

Finally, we derive a bound for $\|I_2\|_{L_2}$.
Since  $T_2(t):= t^{-2} T_1(t^{-1})$ and $T_{2,h}(t):= t^{-2} T_{1,h}(t^{-1})$, we use \eqref{e:tdif} to compute
$$\bal
\pi_H(T_2(t)-T_{2,h}(t))& = t^{-2} \pi_h(T_1(t^{-1})-T_{1,h}(t^{-1}))\\
&=t^{-4} (t^{-2}I+T_h)^{-1}\pi_h (T-T_h) (t^{-2}I+T)^{-1}\\
&= (I+t^2T_h)^{-1} \pi_h (T-T_h) (I+t^2T)^{-1}.
\eal
$$

We clearly have 
$$
\|(I+t^2T_h)^{-1}\pi_h\|_{L_2\to L_2} \leq 1 \quad \hbox{and}\quad
  \|(I+t^2T)^{-1}\|_{L_2 \to L_2} \leq 1.
$$
Combining this with Corollary~\ref{c:apprx} gives
$$\bal
\| \pi_h (T_2(t)-T_{2,h}(t))\|_{L_2} &  \le \big[\|(I+t^2T_h)^{-1}\pi_h\|_{L_2 \to L_2}
\|T-T_h\|_{L_2 \to L_2}\\
&\qquad 
 \|(I+t^2T)^{-1}\|_{L_2\to L_2}\bigg]\preceq
h^{2\alpha}.
\eal
$$
Hence 
$$
\|I_2\|_{L_2} \preceq h^{2\alpha}
\int_{0}^1 t^{1-2\beta} \,dt
  \|f\|_{L_2}
\preceq  h^{2\alpha} \|f\|_{L_2}.
$$
The theorem follows by gathering the estimates for 
$\|I_1\|_{L_2}$ and $\|I_2\|_{L_2}$.
\end{proof}

\begin{remark}
It is only necessary to break the estimate for $I_1$
  into integrals over $[0,t_0]$ and $[t_0,1]$ when $\alpha\ge \beta$ and
$\delta=\gamma$.  In the remainder of the cases, the argument bounding
$[t_0,1]$ can be used to bound the integral over $[0,1]$.
\end{remark}

\begin{remark}[Comparison with \cite{FijitaSuzuki}]  The estimate of \cite{FijitaSuzuki} is that 
$$ \| (T^{\beta  } - T_h^{\beta}\pi_h) f\|_{L_2} 
\leq C h^{2\beta}\|f\|_{L_2}, \Forall f\in L_2,$$
holds when $\alpha=1$.  This is easily obtained from the above 
proof by taking $t_0=h$ and $\delta=0$ in which case
$$
\|I_1\|_{L_2} \le \frac {h^{2\beta}}\beta   \| f \|_{L_2}+C
h^{2\alpha } \int_{h}^1 t^{-1+2\beta-2\alpha}\, dt\, \|f\|_{L_2}
\le C h^{2\beta } \|f\|_{L_2}.
$$
\label{FijitaSuzuki}
\end{remark}

\begin{proof}[Proof of Lemma~\ref{techlem}.]   For the first inequality
  of the lemma, 
  we note that  $0\le \eta <(1+\alpha)/2$ implies 
$0< \theta \le 1$. 
Let $\varphi \in \dH^{2\eta}$ and let $c_i=\langle
\varphi,\tpsi_i\rangle_0$ so that $\varphi=\sum_{i=1}^\infty c_i
\tpsi_i$. 
The definition of the norm in $\dH^{\alpha-1}$ implies that
$$
\| t^{2\theta}  (t^2+T)^{-1} \varphi\|_{\dH^{\alpha-1}}^2 =
\sum_{i=1}^\infty \left( \frac{t^{2\theta}
    \mu_i^{\eta+(1-\alpha)/2}}{t^2+\mu_i} \right)^2
\mu_i^{-2\eta} c_i^2.
$$
Note that  $\theta+\eta+(1-\alpha)/2=1$.    Now $\theta=1$ only if
$\alpha=1$ and $\eta=0$.  In this case, 
the fraction in the parenthesis above  is 
$(t^2/(t^2+\mu_i))$ which is clearly bounded by one.  For $\theta\in
(0,1)$, a Young's inequality with $\frac 1 p=\theta$ and $\frac 1 q = \eta +
(1-\alpha)/2$  gives
$$\frac{t^{2\theta}
    \mu_i^{\eta+(1-\alpha)/2}}{t^2+\mu_i} \le \frac 1 p 
\frac  {t^2}   {t^2+\mu_i} + \frac 1q \frac {\mu_i} {t^2+\mu_i}
\le 1.$$
Thus, 
$$\| t^{2\theta}  (t^2I+T)^{-1} \varphi\|_{\dH^{\alpha-1}}^2 \le 
\sum_{i=1}^\infty \mu_i^{-2\eta} c_i^2 =
\|\varphi\|_{\dH^{2\eta}}^2.
$$
This proves the first inequality of the lemma.

The second inequality is similar.  Indeed,
expanding $f=\sum_{i=1}^M c_{i,h} \tpsi_{i,h}$ with $c_{i,h} = \hzip f{\tpsi_{i,h}} $, we find
$$
A^2:=\| t^{1+\alpha} (t^2I+T_h)^{-1}  f\|^2_{L_2}
=  \sum_{i=1}^M \left(\frac{t^{1+\alpha}}{t^2+\mu_{i,h}}\right)^2c_{i,h}^2.
$$
Thus,
$$
A^2 = \sum_{i=1}^M \left(\frac{t^{1+\alpha} \mu_{i,h}^{\frac{1-\alpha}2}}{t^2+\mu_{i,h}}\right)^2 \mu_{i,h}^{\alpha-1}c_{i,h}^2.
$$
The fraction in the parenthesis is clearly bounded by one when $\alpha=1$.
For $\alpha\in (0,1)$, a Young's inequality with $\frac 1 p = \frac{1+\alpha}{2}$ 
and $\frac 1 q = \frac{1-\alpha}2$ yields 
$$
\left| \frac{t^{1+\alpha} \mu_{i,h}^{\frac{1-\alpha}2}}{t^2+\mu_{i,h}} \right|
\leq \frac 1 p \frac{t^2}{t^2+\mu_{i,h}} + \frac 1 q  \frac{\mu_{i,h}}{t^2+\mu_{i,h}} \leq \frac 1 p + \frac 1 q =1
$$
It follows that $A\le \| f \|_{\dH^{1-\alpha}_h}$ and completes the
proof of the lemma.
\end{proof}

\subsection{Numerical Illustration}

We consider the unit square $\Omega=(0,1)^2 \subset \mathbb R^2$ and study numerically the efficiency in the approximation of $u$ the solution of
$$
(-\Delta)^\beta u = f, \quad \text{in} \quad \Omega, \qquad u=0 \quad \text{on} \quad \partial \Omega
$$
where $f:\Omega \rightarrow \mathbb R$ is given for all $(x_1,x_2)\in \Omega$ by
\begin{equation}\label{e:f}
f(x_1,x_2) = \left\lbrace \begin{array}{ll}
1 &\qquad \text{if} \quad (x_1-0.5)(x_2-0.5)>0\\
0 &\qquad \text{otherwise.}
\end{array}\right.
\end{equation}
Notice that $f\in \dot H^{1/2-\epsilon}(\Omega)$ for all $\epsilon >0$ (but not $\epsilon=0$) and consequently Theorem  \ref{t:tbmtbh} and Remark \ref{dlessamb} predict an error  decay proportional to 
$$
\ln(1/h) h^{2s},
$$ 
with 
$$
s=\left\lbrace \begin{array}{ll}
1, &\quad \text{if} \quad \beta > \frac 34, \\
2(\beta+\frac 14), &\quad \text{otherwise}.
\end{array}\right.
$$
The errors 
$\|u-u^N_h\|_{L_2}$ are computed using the first 300 modes in $x$ and
$y$ (90000 modes) of the Fourier representation of the exact solution
and the number of quadrature points is taken large enough not to
influence the space discretization error.  
Four meshes are used to compute three observed rate of convergence
$$
OROC_i:=\ln\left(e_i/e_{i+1}\right)/\ln\left(h_i/h_{i+1}\right), \quad i=1,2,3,
$$
where $e_i$ are the $L_2$ errors and $h_i$ the diameter of the quasi-uniform subdivision $i$. 
The average of the three observed rate of convergence $AROC$ are reported in Table \ref{t:observed}
and are comparable with the rate predicted by Theorem \ref{t:tbmtbh}
(see, Remark~\ref{dlessamb}).
\begin{table}[h!]
\setlength{\extrarowheight}{1ex}
\begin{center}
\begin{tabular}{c|c|c|c|c|c|c|c||c|c}
  & \multicolumn{7}{c||}{$\beta < \frac 34$} & \multicolumn{2}{c}{$\beta > \frac 34$}\\
  \cline{2-10}
  & 0.1 &0.2 & 0.3 & 0.4 & 0.5 & 0.6 & 0.7 & 0.8 & 0.9 \\
\hline
AROC & 0.92 & 1.06 & 1.22   & 1.4  &  1.52    & 1.72    & 1.86    & 1.94     & 1.96 \\
THM & 0.7 & 0.9 & 1.1   & 1.3 & 1.5 & 1.7 & 1.9  & 2.0 & 2.0 \\
\end{tabular}
\caption{Average observed rate of convergence (AROC) for different values of $\beta$ compared with
  the rates predicted by Theorem  \ref{t:tbmtbh} (THM) ignoring the log term.}
\end{center}
\label{t:observed}
\end{table}

Figure \ref{f:conv} depicts the decay of the error $\|u-u^N_h\|_{L_2}$ for different values of $\beta$.
\begin{figure}
\centerline{\includegraphics[width=0.75\textwidth]{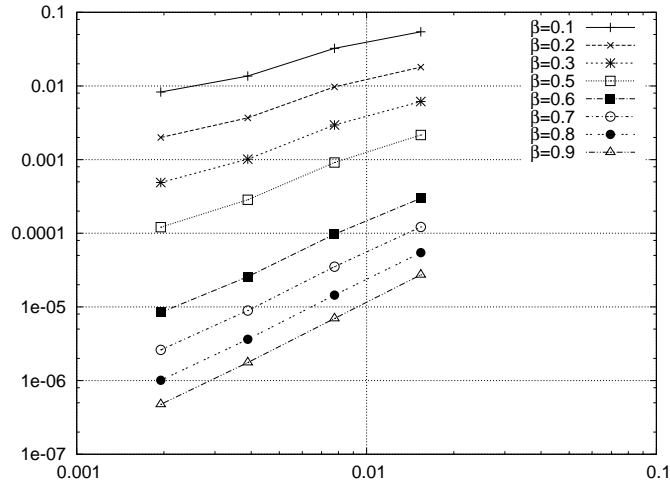}}
\caption{$L_2$-errors for different values of $\beta$ versus the approximate mesh size diameter used for the computations. 
The number of quadrature points is taken large enough not to influence the errors. 
The corresponding average observed rate of convergence are reported in Table \ref{t:observed}.}
\label{f:conv}
\end{figure}

The solution for $\beta=0.1$, $0.5$ and $0.8$ are depicted in figure \ref{f:solutionchecker} together with a cut over the horizontal line $\lbrace y=0.25 \rbrace$.
\begin{figure}
\begin{center}
\begin{tabular}{cc}
\includegraphics[width=0.45\textwidth]{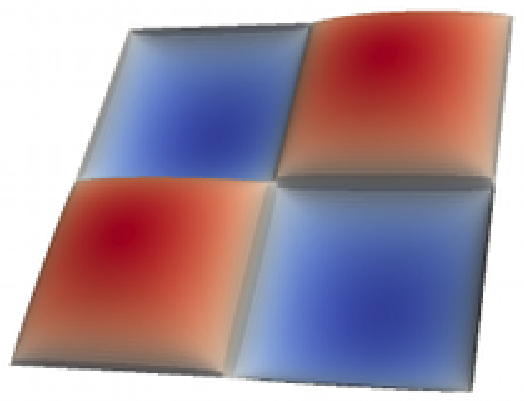}&
\includegraphics[width=0.45\textwidth]{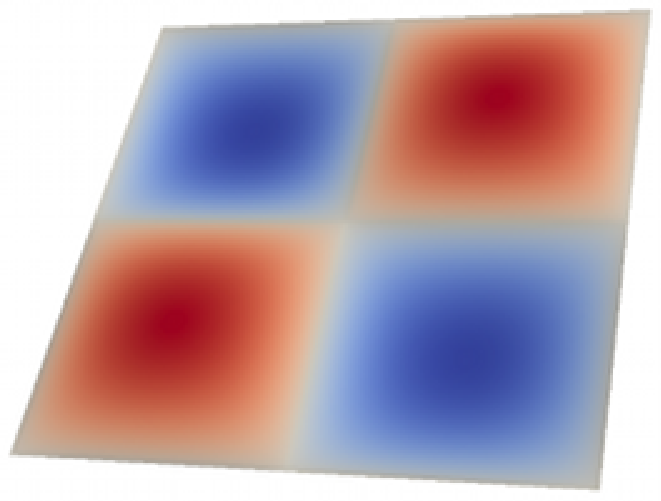}\\
\includegraphics[width=0.45\textwidth]{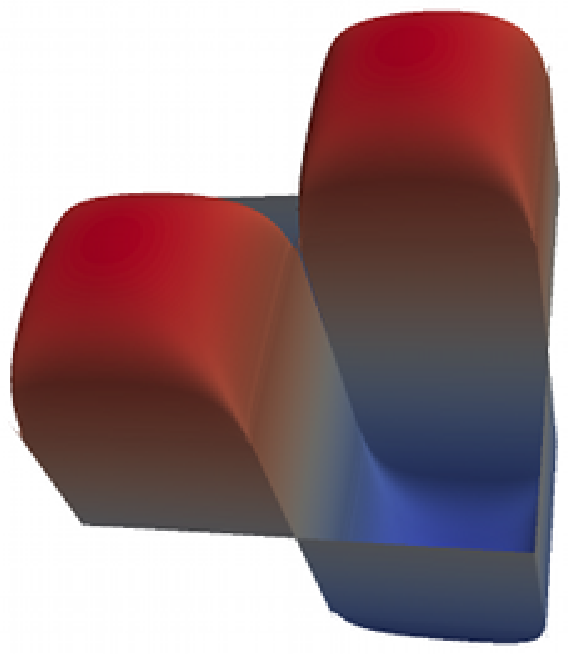}&
\includegraphics[width=0.45\textwidth]{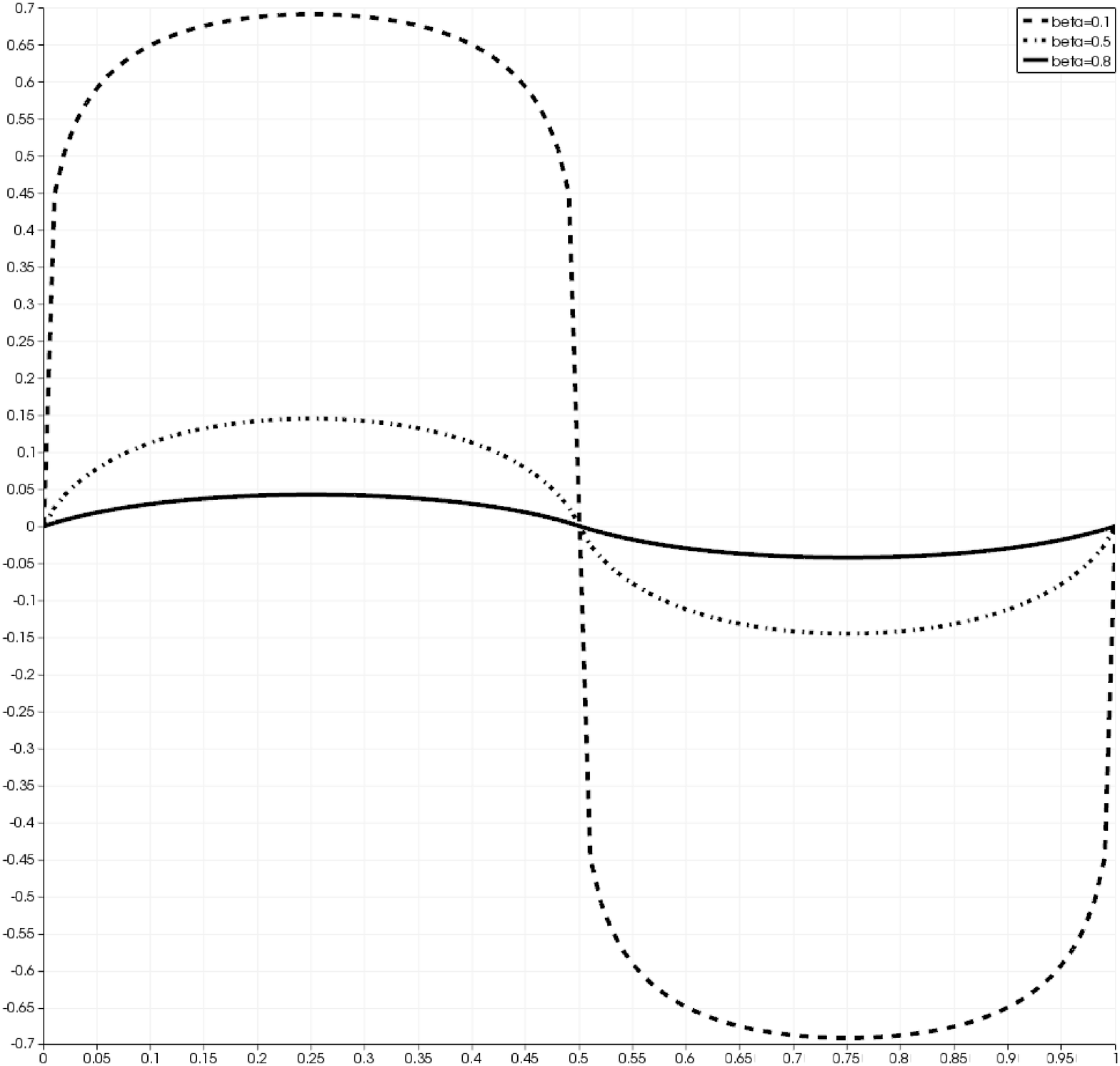}
\end{tabular}
\end{center}
\caption{Values of $(-\Delta)^{-\beta} f$ when $f$ is the checkerboard function \eqref{e:f} for $\beta=0.5$ (top left), $\beta=0.8$ (top right) and $\beta=0.1$ (bottom left).
The difference in shape and magnitude is highlighted in the bottom right plots which depicts the values of these function along the line $\lbrace y=0.25\rbrace$ ($\beta=0.1$ dashed line, $\beta=0.5$ doted line, $\beta=0.8$ solid line).}
\label{f:solutionchecker}
\end{figure}

%
%
%

\section*{Acknowledgment}
This work was supported in part
by award number KUS-C1-016-04 made by King Abdulla University of
  Science and Technology (KAUST).
The first author was
also  supported in part by
  the National Science Foundation through Grant DMS-1254618 while the 
second  was also supported in part by
  the National Science Foundation through Grant DMS-1216551.
In addition, the authors would like to thank E. Otarola, R.H. Nochetto and A. Salgado for fruitful discussions on the subject.

\bibliographystyle{plain}

\begin{thebibliography}{10}

\bibitem{bramblebacuta}
Constantin Bacuta, James~H. Bramble, and Joseph~E. Pasciak.
\newblock New interpolation results and applications to finite element methods
  for elliptic boundary value problems.
\newblock {\em East-West J. Numer. Math.}, 3:179--198, 2001.

\bibitem{MR2450437}
Oleg~G. Bakunin.
\newblock {\em Turbulence and diffusion}.
\newblock Springer Series in Synergetics. Springer-Verlag, Berlin, 2008.
\newblock Scaling versus equations.

\bibitem{balakrishnan}
V.~Balakrishnan.
\newblock Fractional powers of closed operators and the semi-groups generated
  by them.
\newblock {\em Pacific J. Math.}, 10:419--437, 1960.

\bibitem{MR2223347}
Peter~W. Bates.
\newblock On some nonlocal evolution equations arising in materials science.
\newblock In {\em Nonlinear dynamics and evolution equations}, volume~48 of
  {\em Fields Inst. Commun.}, pages 13--52. Amer. Math. Soc., Providence, RI,
  2006.

\bibitem{MR1192782}
M.~Sh. Birman and M.~Z. Solomjak.
\newblock {\em Spectral theory of selfadjoint operators in {H}ilbert space}.
\newblock Mathematics and its Applications (Soviet Series). D. Reidel
  Publishing Co., Dordrecht, 1987.
\newblock Translated from the 1980 Russian original by S. Khrushch{\"e}v and V.
  Peller.

\bibitem{00BZ}
J.H. Bramble and X.~Zhang.
\newblock The analysis of multigrid methods.
\newblock In {\em Handbook of numerical analysis, {V}ol. {VII}}, Handb. Numer.
  Anal., VII, pages 173--415. North-Holland, Amsterdam, 2000.

\bibitem{MR2354493}
Luis Caffarelli and Luis Silvestre.
\newblock An extension problem related to the fractional {L}aplacian.
\newblock {\em Comm. Partial Differential Equations}, 32(7-9):1245--1260, 2007.

\bibitem{ISI:000175019600004}
P~Carr, H~Geman, DB~Madan, and M~Yor.
\newblock {The fine structure of asset returns: An empirical investigation}.
\newblock {\em {JOURNAL OF BUSINESS}}, {75}({2}):{305--332}, {APR} {2002}.

\bibitem{MR1709781}
Peter Constantin and Jiahong Wu.
\newblock Behavior of solutions of 2{D} quasi-geostrophic equations.
\newblock {\em SIAM J. Math. Anal.}, 30(5):937--948, 1999.

\bibitem{dauge}
M.~Dauge.
\newblock {\em Elliptic Boundary Value Problems on Corner Domains}.
\newblock Lecture Notes in Mathematics, 1341, Springer-Verlag, 1988.

\bibitem{MR0521262}
G.~Duvaut and J.-L. Lions.
\newblock {\em Inequalities in mechanics and physics}.
\newblock Springer-Verlag, Berlin, 1976.
\newblock Translated from the French by C. W. John, Grundlehren der
  Mathematischen Wissenschaften, 219.

\bibitem{MR1918950}
A.~Cemal Eringen.
\newblock {\em Nonlocal continuum field theories}.
\newblock Springer-Verlag, New York, 2002.

\bibitem{FijitaSuzuki}
Firoshi Fujita and Takashi Suzuki.
\newblock {\em Evolution Problems}.
\newblock Handbook of Numerical Analysis, II, Finite element methods. Part 1
  (Ciarlet, P. G. and Lions, J.-L., eds.). North-Holland, Amsterdam, 1991.

\bibitem{Gavrilyuk1}
Ivan~P. Gavrilyuk.
\newblock An algorithmic representation of fractional powers of positive
  operators.
\newblock {\em Numer. Funct. Anal. Optim.}, 17(3-4):293--305, 1996.

\bibitem{hackbush}
Ivan~P. Gavrilyuk, Wolfgang Hackbusch, and Boris~N. Khoromskij.
\newblock H-matrix approximation for the operator exponential with
  applications.
\newblock {\em Numerische Mathematik}, 92:83--111, 2002.

\bibitem{GHK2004}
Ivan~P. Gavrilyuk, Wolfgang Hackbusch, and Boris~N. Khoromskij.
\newblock Data-sparse approximation to the operator-valued functions of
  elliptic operator.
\newblock {\em Math. Comp.}, 73(247):1297--1324, 2004.

\bibitem{Gavrilyuk3}
Ivan~P. Gavrilyuk, Wolfgang Hackbusch, and Boris~N. Khoromskij.
\newblock Data-sparse approximation to a class of operator-valued functions.
\newblock {\em Math. Comp.}, 74(250):681--708, 2005.

\bibitem{Gavrilyuk2}
Ivan~P. Gavrilyuk, Wolfgang Hackbusch, and Boris~N. Khoromskij.
\newblock Hierarchical tensor-product approximation to the inverse and related
  operators for high-dimensional elliptic problems.
\newblock {\em Computing}, 74(2):131--157, 2005.

\bibitem{MR2480109}
Guy Gilboa and Stanley Osher.
\newblock Nonlocal operators with applications to image processing.
\newblock {\em Multiscale Model. Simul.}, 7(3):1005--1028, 2008.

\bibitem{guermondima2009}
J.-L. Guermond.
\newblock The {LBB} condition in fractional {S}obolev spaces and applications.
\newblock {\em IMA J. Numer. Anal.}, 29(3):790--805, 2009.

\bibitem{MR2252038}
M.~Ilic, F.~Liu, I.~Turner, and V.~Anh.
\newblock Numerical approximation of a fractional-in-space diffusion equation.
  {I}.
\newblock {\em Fract. Calc. Appl. Anal.}, 8(3):323--341, 2005.

\bibitem{MR2300467}
M.~Ilic, F.~Liu, I.~Turner, and V.~Anh.
\newblock Numerical approximation of a fractional-in-space diffusion equation.
  {II}. {W}ith nonhomogeneous boundary conditions.
\newblock {\em Fract. Calc. Appl. Anal.}, 9(4):333--349, 2006.

\bibitem{kellogg}
R.B. Kellogg.
\newblock Interpolation between subspaces of a hilbert space.
\newblock Technical report, Univ. of Maryland,, Inst. Fluid Dynamics and App.
  Math., Tech. Note BN-719, 1971.

\bibitem{lundbowers}
John Lund and Kenneth~L. Bowers.
\newblock {\em Sinc methods for quadrature and differential equations}.
\newblock Society for Industrial and Applied Mathematics (SIAM), Philadelphia,
  PA, 1992.

\bibitem{MR660727}
B.~M. McCay and M.~N.~L. Narasimhan.
\newblock Theory of nonlocal electromagnetic fluids.
\newblock {\em Arch. Mech. (Arch. Mech. Stos.)}, 33(3):365--384, 1981.

\bibitem{macleanthomee}
William McLean and Vidar Thom{\'e}e.
\newblock Iterative solution of shifted positive-definite linear systems
  arising in a numerical method for the heat equation based on {L}aplace
  transformation and quadrature.
\newblock {\em ANZIAM J.}, 53(2):134--155, 2011.

\bibitem{others}
Sergey Nazarov and Boris Plamenevsky.
\newblock {\em Elliptic problems in domains with piecewise smooth boundaries}.
\newblock De Gruyter expositions in mathematics, De Gruyter, 1994.

\bibitem{Abner}
R.H. Nochetto, E.~Otarola, and A.J. Salgado.
\newblock A {PDE} approach to fractional diffusion in general domains. part
  {I}: A priori error analysis.
\newblock submitted.

\bibitem{MR1727557}
S.~A. Silling.
\newblock Reformulation of elasticity theory for discontinuities and long-range
  forces.
\newblock {\em J. Mech. Phys. Solids}, 48(1):175--209, 2000.

\bibitem{thomee}
Vidar Thom\'ee.
\newblock {\em {G}alerkin finite element methods for parabolic problems}.
\newblock Springer series in computational mathematics. Springer, New York,
  1997.

\bibitem{MR2800568}
Qianqian Yang, Ian Turner, Fawang Liu, and Milos Ili{\'c}.
\newblock Novel numerical methods for solving the time-space fractional
  diffusion equation in two dimensions.
\newblock {\em SIAM J. Sci. Comput.}, 33(3):1159--1180, 2011.

\bibitem{yoshida}
Toshio Yoshida.
\newblock {\em Functional Analysis}.
\newblock Springer-Verlag, New York, 1995.

\end{thebibliography}

\end{document}